\documentclass[12pt]{article}
\usepackage{amsmath}
\usepackage{amssymb,amscd}
\usepackage{bm}
\usepackage{wrapfig}

\usepackage{amsbsy} 
\usepackage{hyperref} 

\usepackage{color}
\usepackage{bm}
\usepackage[all]{xy}

\usepackage{theorem}

\newtheorem{theorem}{Theorem}[section]
 \newtheorem{proposition}[theorem]{Proposition}
 
 \newtheorem{lemma}[theorem]{Lemma}

{\theorembodyfont{\normalfont}
 \newtheorem{definition}[theorem]{Definition}
 \newtheorem{example}{Example}[section]
 \newtheorem{remark}{Remark}[section]
 }

\newenvironment{proof}[1][Proof]{\begin{trivlist}
\item[\hskip \labelsep {\bfseries #1}]}{\end{trivlist}}
\newcommand{\qed}{\nobreak \ifvmode \relax \else
      \ifdim\lastskip<1.5em \hskip-\lastskip
      \hskip1.5em plus0em minus0.5em \fi \nobreak
      \vrule height0.75em width0.5em depth0.25em\fi}


\newlength{\minitwocolumn}
\newcommand{\beq}{\begin{equation}}
\newcommand{\eeq}{\end{equation}}
\newcommand{\bea}{\begin{eqnarray*}}
\newcommand{\eea}{\end{eqnarray*}}
\newcommand{\beqa}{\begin{eqnarray}}
\newcommand{\eeqa}{\end{eqnarray}}

\newcommand{\bX}{\boldsymbol{X}}

\def\bw{{\bm{W}}}
\def\bW{{\bm{W}}}
\def\by{{\bm{Y}}}
\def\bz{{\bm{Z}}}

\def\bvartheta{{\bm{\vartheta}}}

\def\bR{{\mathbb{R}}}

\def\brT{{\mathbb{T}}}

\newcommand{\calD}{{\mathcal D}}

\newcommand{\calF}{{\mathcal F}}

\newcommand{\calL}{{\mathcal L}}
\newcommand{\calM}{{\mathcal M}}

\newcommand{\calO}{{\mathcal O}}

\newcommand{\calR}{{\mathcal R}}

\newcommand{\Map}{{\rm Map}}
\newcommand{\ev}{{\rm ev}}

\newcommand{\sbv}[2]{{\{{{#1},{#2}}\}}}
\newcommand{\courant}[2]{{[{{#1},{#2}}]_D}}

\newcommand{\courantr}[2]{{[{{#1},{#2}}]_D^{\pi}}}

\newcommand{\rhoe}{\rho_E}
\newcommand{\rhoa}{{\rho_A}}

\newcommand{\rhott}{\rho_{T \oplus T^*}}

\newcommand{\bracket}[2]{\langle #1,\,#2\rangle}
\newcommand{\bracketm}[2]{{\langle #1,\,#2\rangle_{-}}}

\newcommand{\inner}[2]{{({{#1},{#2}})}}

\newcommand\Hom{\mathop{\mathrm{Hom}}\nolimits}

\newcommand{\rd}{\mathrm{d}}
\def\sd{{\bm{\mathrm{d}}}}
\def\qd{\mathrm{D}}
\def\sdd{{\bm{\mathrm{D}}}}
\def\bbd{{\bm{\mathrm{d}}}}

\newcommand{\momega}{{\omega}}
\newcommand{\gomega}{{\omega}_{grad}}
\newcommand{\ggomega}{{\omega_{AKSZ}}}
\newcommand{\gS}{{S_{AKSZ}}}

\newcommand{\baS}{{{}^A S}}
\newcommand{\rdd}{{\mathrm{D}}}

\newcommand{\rdv}{{\mathrm{D}_V}}
\newcommand{\rdvn}{{\mathrm{D}_V^{\nabla}}}

\newcommand{\spans}[1]{\langle #1 \rangle}

\begin{document}


\baselineskip 0.7cm

\begin{titlepage}
\begin{flushright}
\end{flushright}

\vskip 1.35cm
\begin{center}
{\Large \bf
Hamilton Lie algebroids over Dirac structures
and sigma models}
\vskip 1.2cm
Noriaki Ikeda
\footnote{E-mail:\
nikeda@se.ritsumei.ac.jp
}
\vskip 0.4cm

{\it
Department of Mathematical Sciences,
Ritsumeikan University \\
Kusatsu, Shiga 525-8577, Japan \\
}
\vskip 0.4cm

\today

\vskip 1.5cm

\begin{abstract}
We propose a Hamiltonian Lie algebroid and a momentum section over a Dirac structure as a generalization of a Hamiltonian Lie algebroid over a pre-symplectic manifold and one over a Poisson manifold.
A Hamiltonian Lie algebroid and a momentum section generalize a Hamiltonian 
$G$-space and a momentum map over a symplectic manifold.
We prove some properties of this new Hamiltonian Lie algebroid and construct a mechanics based on this structure as an application. These new mechanics are called the gauged Poisson sigma model and the gauged Dirac sigma model.
\end{abstract}

\end{center}
\end{titlepage}



\setcounter{page}{2}


\rm

\section{Introduction}
\noindent
A momentum map is a fundamental object in symplectic geometry, inspired by symmetries and conserved quantities in analytical mechanics.
A momentum map is defined on a symplectic manifold with a Lie group action.
Moreover, we assume the existence of a so-called Hamiltonian function, which generates the corresponding action. Such a space is called a Hamiltonian $G$-space.

Analyses of physical models suggest that Lie group actions in momentum maps should be generalized to Lie groupoid actions in order to capture symmetries and conserved quantities in physical theories \cite{Cattaneo:2000iw, Alekseev:2004np, Blohmann:2010jd}.
A 'groupoid' generalization of a Lie algebra is a Lie algebroid.

A Lie algebroid is an infinitesimal object of a Lie groupoid, analogous to the way a Lie algebra is an infinitesimal object of a Lie group.

Recently, Blohmann and Weinstein \cite{Blohmann:2018} proposed a generalization of a momentum map and a Hamiltonian $G$-space over a pre-symplectic manifold with a Lie algebra (Lie group) action to a Lie algebroid (Lie groupoid) setting. Their work was inspired by the analysis of the Hamiltonian formalism of general relativity \cite{Blohmann:2010jd}. Mathematically, the idea of generalizing a momentum map is natural in the following sense.
A momentum map is a map $\mu: M \rightarrow \mathfrak{g}^*$, where $\mathfrak{g}^*$ is the dual of a Lie algebra. It is regarded as a section of a trivial vector bundle $M \times \mathfrak{g}^*$.
It is natural to generalize the trivial bundle to the dual of a general vector bundle $A^*$ and consider a corresponding generalization of momentum maps.

The generalization of the momentum map for a Lie algebroid is called a momentum section, and a Hamiltonian $G$-space is generalized to a Hamiltonian Lie algebroid \cite{Blohmann:2018}.
Moreover, the concept of a momentum section has been further extended to a momentum section on a Courant algebroid \cite{Ikeda:2021fjk} and to momentum sections over a pre-multisymplectic manifold \cite{Hirota:2021isx}.

A momentum section and a Hamiltonian Lie algebroid over a Poisson manifold were proposed by Blohmann, Ronchi, and Weinstein in \cite{Blohmann:2023}.

In this paper, we present a unified formulation of Hamiltonian Lie algebroids over pre-symplectic manifolds and Poisson manifolds.
Both pre-symplectic structures and Poisson structures are described as 
Dirac structures \cite{Courant}, which are subbundles of Courant algebroids $TM \oplus T^*M$ \cite{LWX}.
Relations between momentum maps on symplectic manifolds and Dirac structures have been studied in
\cite{BursztynCrainic1, BursztynCrainic2, BursztynIglesiasPonteSevera, Balibanu-Mayrand}.
We define and analyze a Hamiltonian Lie algebroid and a momentum section on a Dirac structure, and we prove that Hamiltonian Lie algebroids over pre-symplectic and Poisson structures are special cases of this framework.

Some physically relevant models, such as constrained Hamiltonian mechanics and sigma models, exhibit this structure \cite{Ikeda:2019pef}.
In Section \ref{sec:physmod}, we construct physical models incorporating the structures of Hamiltonian Lie algebroids over Poisson manifolds and Dirac structures, respectively. These models are topological sigma models that generalize the Poisson sigma model and the Dirac sigma model.
We refer to them as the gauged Poisson sigma model (GPSM) and the gauged Dirac sigma model, which can be understood as Poisson sigma models and Dirac sigma models equipped with Lie algebroid gauge symmetries.
Our construction is based on the Q-manifold and the AKSZ formalism \cite{Alexandrov:1995kv}. See also \cite{Cattaneo:2001ys, Roytenberg:2006qz, Ikeda:2012pv}.

For momentum maps with Lie group actions, the BFV and AKSZ formulations have been studied in \cite{Bonechi:2012kh}, and the gauging of the Poisson sigma model has been analyzed in \cite{Zucchini:2008cg}.
Our realization of the BV action differs from previous approaches, as we construct a Q-manifold formulation of the Hamiltonian Lie algebroid and the momentum section.

This paper is organized as follows.
In Section 2, we introduce Lie algebroids and related concepts, and we establish the notation used throughout the paper.
In Section 3, we define a Hamiltonian Lie algebroid and a momentum section over a Dirac structure and discuss some of their properties.
We prove that Hamiltonian Lie algebroids over pre-symplectic and Poisson manifolds are special cases of this framework.
In Section 4, we present some examples.
In Section 5, we construct sigma models that incorporate Hamiltonian Lie algebroids over Poisson manifolds and Dirac structures.
In Section 6, we provide a discussion and outlook.
Finally, in the Appendix, we summarize some useful formulas.

\section{Preliminary}
\noindent
In this section, we summarize some definitions and previous results, including Lie algebroids and related notions, momentum sections, and Hamiltonian Lie algebroids over pre-symplectic and Poisson manifolds.

\subsection{Lie algebroids}

\begin{definition}
Let $A$ be a vector bundle over a smooth manifold $M$.
A Lie algebroid $(A, [-,-], \rhoa)$ is a vector bundle $A$ with
a bundle map $\rhoa: A \rightarrow TM$ called the anchor map, 
and a Lie bracket
$[-,-]: \Gamma(A) \times \Gamma(A) \rightarrow \Gamma(A)$
satisfying the Leibniz rule,
\begin{eqnarray}
[e_1, fe_2] &=& f [e_1, e_2] + \rhoa(e_1) f \cdot e_2,
\end{eqnarray}
where $e_1, e_2 \in \Gamma(A)$ and $f \in C^{\infty}(M)$.
\end{definition}
A Lie algebroid generalizes both a Lie algebra and the space of vector fields on a smooth manifold.
\begin{example}[Lie algebras]
Let a manifold $M$ be one point $M = \{pt \}$. 
Then a Lie algebroid is a Lie algebra $\mathfrak{g}$.
\end{example}
\begin{example}[Tangent Lie algebroids]\label{tangentLA}
If a vector bundle $A$ is a tangent bundle $TM$ and $\rhoa = \mathrm{id}$, 
then a bracket $[-,-]$ is a normal Lie bracket 
on the space of vector fields $\mathfrak{X}(M)$
and $(TM, [-,-], \mathrm{id})$ is a Lie algebroid.
It is called a \textit{tangent Lie algebroid}.
\end{example}

\begin{example}[Action Lie algebroids]\label{actionLA}
Assume a smooth action of a Lie group $G$ to a smooth manifold $M$, 
{$M \times G \rightarrow M$.}
The differential map of this action induces an infinitesimal action of the Lie algebra $\mathfrak{g}$ of $G$ on $M$.
Since $\mathfrak{g}$ acts as a differential operator on $M$,
the differential map
is a bundle map $\rho: M \times \mathfrak{g} \rightarrow TM$.
Consistency of a Lie bracket requires that $\rho$ is 
a Lie algebra morphism such that
\begin{eqnarray}
~[\rho(e_1), \rho(e_2)] &=& \rho([e_1, e_2]),
\label{almostLA}
\end{eqnarray}
where the bracket on the left-hand side 
is the standard Lie bracket of vector fields.  
These data define a Lie algebroid $(A= M \times \mathfrak{g}, [-,-], \rho)$.
known as an \textit{action Lie algebroid}.
\end{example}

\begin{example}[Poisson Lie algebroids]\label{Poisson}
%
A bivector field $\pi \in \Gamma(\wedge^2 TM)$ is called a Poisson bivector field if $[\pi, \pi]_S =0$, where $[-,-]_S$ denotes the Schouten bracket on the space of multivector fields, $\Gamma(\wedge^{\bullet} TM)$.
A smooth manifold $M$ equipped with a Poisson bivector field $\pi$ is 
called a Poisson manifold $(M, \pi)$.

Let $(M, \pi)$ be a Poisson manifold. Then, a Lie algebroid structure is induced on $T^*M$.
The map $\pi^{\sharp}$ is defined by
$\pi^{\sharp}: T^*M \rightarrow TM$ by $\bracket{\pi^{\sharp}(\alpha)}{\beta}
= \pi(\alpha, \beta)$ for all $\beta \in \Omega^1(M)$.
The anchor map is given by $\rho= - \pi^{\sharp}$,
and a Lie bracket on $\Omega^1(M)$ is defined by the Koszul bracket:
\begin{eqnarray}
[\alpha, \beta]_{\pi} = \calL_{\pi^{\sharp} (\alpha)}\beta - \calL_{\pi^{\sharp} (\beta)} \alpha - \rd(\pi(\alpha, \beta)),
\label{Koszulbracket}
\end{eqnarray}
where $\alpha, \beta \in \Omega^1(M)$.
Thus, $(T^*M, [-, -]_{\pi}, -\pi^{\sharp})$ is a Lie algebroid.
\end{example}

\begin{example}[Lie algebroids induced from twisted Poisson structures]\label{tPoisson}
A bivector field $\pi \in \Gamma(\wedge^2 TM)$ and 
a closed $3$-form $H \in \Omega^3(M)$ define a twisted Poisson manifold
$(M, \pi, H)$ if they satisfy
\begin{eqnarray}
&& \frac{1}{2}[\pi, \pi]_S 
= \bracket{\otimes^{3} \pi}{H},
\label{tPoisson1}
\\
&& \rd H =0.
\end{eqnarray}
Here, the term $\bracket{\otimes^{3} \pi}{H}$ is given by 
\begin{eqnarray}
\bracket{\otimes^{3} \pi}{H}(\alpha_1, \alpha_2, \alpha_3)
:= H(\pi^{\sharp} (\alpha_1), \pi^{\sharp} (\alpha_2), \pi^{\sharp} (\alpha_3)),
\end{eqnarray}
for $\alpha_i \in \Omega^1(M)$ \cite{Klimcik:2001vg, Severa:2001qm}.

As in the Poisson case, we define the bundle map, 
$\rho= -\pi^{\sharp}: T^*M \rightarrow TM$,
and introduce a Lie bracket deformed by $H$:
\begin{eqnarray}
[\alpha, \beta]_{\pi,H} = \calL_{\pi^{\sharp} (\alpha)}\beta - \calL_{\pi^{\sharp} (\beta)} \alpha - \rd(\pi(\alpha, \beta))
+ \iota_{\pi^{\sharp}(\alpha)} \iota_{\pi^{\sharp}(\beta)} H,
\end{eqnarray}
for $\alpha, \beta \in \Omega^1(M)$.
Then, $(T^*M, [-, -]_{\pi, H}, -\pi^{\sharp})$ defines a Lie algebroid.
\end{example}
For a comprehensive review of Lie algebroids and their properties, see, for example, \cite{Mackenzie}.

For a Lie algebroid $A$, sections of the exterior algebra of $A^*$ are called \textit{$A$-differential forms}.
A differential ${}^A \rd: \Gamma(\wedge^m A^*)
\rightarrow \Gamma(\wedge^{m+1} A^*)$ on the spaces of $A$-differential forms, $\Gamma(\wedge^{\bullet} A^*)$,
called the \textit{Lie algebroid differential}, 
or the \textit{$A$-differential}, is defined as follows. 
\begin{definition}
The Lie algebroid differential ${}^A \rd: \Gamma(\wedge^m A^*)
\rightarrow \Gamma(\wedge^{m+1} A^*)$ is given by
\begin{eqnarray}
{}^A \rd \eta(e_1, \ldots, e_{m+1}) 
&=& \sum_{i=1}^{m+1} (-1)^{i-1} \rho(e_i) \eta(e_1, \ldots, 
\check{e_i}, \ldots, e_{m+1})
\nonumber \\ && 
+ \sum_{1 \leq i < j \leq m+1} (-1)^{i+j} \eta([e_i, e_j], e_1, \ldots, \check{e_i}, \ldots, \check{e_j}, \ldots, e_{m+1}),
\label{LAdifferential}
\end{eqnarray}
where $\eta \in \Gamma(\wedge^m A^*)$ and $e_i \in \Gamma(A)$.
\end{definition}
The $A$-differential satisfies $({}^A \rd)^2=0$.
It generalizes both the de Rham differential on $T^*M$ and the Chevalley-Eilenberg differential on a Lie algebra.

\subsection{Connections on Lie algebroids}\label{connectionLA}
We begin by introducing a connection on a vector bundle $E$.
A connection is an $\bR$-linear map,
$\nabla:\Gamma(E)\rightarrow \Gamma(E \otimes T^*M)$,
satisfying the Leibniz rule,
\begin{eqnarray}
\nabla (f s) = f \nabla s + (\rd f) \otimes s,
\end{eqnarray}
for all $s \in \Gamma(E)$ and $f \in C^{\infty}(M)$.
The dual connection on $E^*$ is defined by the equation,
\begin{eqnarray}
\rd \inner{\mu}{s} = \inner{\nabla \mu}{s} + \inner{\mu}{\nabla s},
\end{eqnarray}
for all sections $\mu \in \Gamma(E^*)$ and $s \in \Gamma(E)$,
where $\inner{-}{-}$ denotes the pairing between $E$ and $E^*$.
For simplicity, we use the same notation $\nabla$ for the dual connection.

On a Lie algebroid, we define another derivation called an $A$-connection.
Let $A$ be a Lie algebroid over a smooth manifold $M$, and let $E$ be a vector bundle over the same base manifold $M$.
An \textit{$A$-connection} on $E$ 
with respect to the Lie algebroid $A$ is a $\bR$-linear map,
${}^A \nabla: \Gamma(E) \rightarrow \Gamma(E \otimes A^*)$,
satisfying 
\begin{eqnarray}
{}^A \nabla_e (f s) = f {}^A \nabla_e s + (\rho(e) f) s,
\end{eqnarray}
for $e \in \Gamma(A)$, $s \in \Gamma(E)$ and $f \in C^{\infty}(M)$.
%
The ordinary connection on $TM$ is a special case of an 
$A$-connection for $A=TM$, the ordinary connection 
$\nabla$ corresponds to the $A$-connection ${}^{TM} \nabla$.

If a standard vector bundle connection $\nabla$ on $A$ is given,
then, an $A$-connection for the tangent bundle $E=TM$,
${}^A \nabla: \Gamma(TM) \rightarrow \Gamma(TM \otimes A^*)$,
called the \textit{basic $A$-connection}
is defined by
\begin{eqnarray}
{}^A \nabla_{e} v &:=& \calL_{\rho(e)} v + \rho(\nabla_v e)
= [\rho(e), v] + \rho(\nabla_v e),
\label{stEconnection1}
\end{eqnarray}
where 
$e \in \Gamma(A)$ and $v \in \mathfrak{X}(M)$.
For a $1$-form $\alpha \in \Omega^1(M)$, the \textit{basic $A$-connection} 
is given by
\begin{eqnarray}
{}^A \nabla_{e} \alpha &:=& \calL_{\rho(e)} \alpha 
+ \bracket{\rho(\nabla e)}{\alpha}.
\label{Econoneform}
\end{eqnarray}
For a general discussion on connections on Lie algebroids, see
\cite{AbadCrainic, CrainicFernandes, DufourZung}.

Given a connection $\nabla$, the covariant derivative extends naturally  
to the derivation on the space of differential forms taking a value in
$\wedge^m A^*$,
$\Omega^k(M, \wedge^m A^*)$, which is called 
the \textit{exterior covariant derivative}.

Similarly, an $A$-connection ${}^A \nabla$ extends to a derivation
satisfying the Leibniz rule on the space 
$\Gamma(\wedge^m A^* \otimes \wedge^k E)$.
This extension is called \textit{the $A$-exterior covariant derivative}
${}^A \nabla: \Gamma(\wedge^m A^* \otimes \wedge^k E)
\rightarrow \Gamma(\wedge^{m+1} A^* \otimes \wedge^k E)$
and is denoted by the same notation $\nabla$ and ${}^A\nabla$.
In this paper, we consider the $A$-exterior covariant derivatives 
on $E = T^*M$ and $E = TM$.
For $E = T^*M$, the concrete definition is as follows.
\begin{definition}
For $\Omega^k(M, \wedge^m A) = \Gamma(\wedge^m A^* \otimes \wedge^k T^*M)$,
the \textit{$A$-exterior covariant derivative}
${}^A \nabla: \Omega^k(M, \wedge^m A^*) \rightarrow \Omega^k(M, \wedge^{m+1} A^*)$ is defined by
\begin{eqnarray}
({}^A \nabla \alpha)(e_1, \ldots, e_{m+1})
&:=& \sum_{i=1}^{m+1} (-1)^{i-1} 
{}^A \nabla_{e_i}
(\alpha(e_1, \ldots, \check{e_i}, \ldots, e_{m+1}))
\nonumber \\ && 
+ \sum_{1 \leq i < j \leq m+1} (-1)^{i+j} \alpha([e_i, e_j], e_1, \ldots, \check{e_i}, \ldots, \check{e_j}, \ldots, e_{m+1}),
\label{Ediffconnection}
\end{eqnarray}
for $\alpha \in \Omega^k(M, \wedge^m A^*)$ and $e_i \in \Gamma(A)$.
\end{definition}
For $E=TM$, the $A$-exterior covariant derivative is given by
taking the pairing,
\begin{eqnarray}
{}^A \rd \bracket{\phi}{\alpha_1, \ldots, \alpha_k} = 
\bracket{{}^A \nabla \phi}{\alpha_1, \ldots, \alpha_k}
+ \sum_{i=1}^k (-1)^{i-1} \bracket{\phi}{\alpha_1, \ldots, {}^A \nabla \alpha_i, \ldots, \alpha_k},
\end{eqnarray}
for $\phi \in \mathfrak{X}^k(M, \wedge^m A^*)$
and $\alpha_i \in \Omega^1(M)$.

For two connections, $\nabla$ and ${}^A \nabla$, 
we introduce tensors:
the \textit{curvature}, $R \in \Omega^2(M, A \otimes A^*)$, 
the \textit{$A$-torsion}, $T \in \Gamma(A \otimes \wedge^2 A^*)$, 
and the \textit{basic curvature} \cite{Blaom}.
$\baS \in \Omega^1(M, \wedge^2 A^* \otimes A)$.
These are defined as follows:
\beqa
R(e, e^{\prime}) &:=& [\nabla_e, \nabla_{e^{\prime}}] - \nabla_{[e, e^{\prime}]}, 
\label{curv}
\\
T(e, e^{\prime}) &:=& 
\nabla_{\rho(e)} e^{\prime} - \nabla_{\rho(e^{\prime})} e - [e, e^{\prime}],
\label{Etorsion}
\\
\baS(e, e^{\prime}) &:=& \calL_e (\nabla e^{\prime}) 
- \calL_{e^{\prime}} (\nabla e) 
- \nabla_{\rho(\nabla e)} e^{\prime} + \nabla_{\rho(\nabla e^{\prime})} e
\nonumber \\ &&
- \nabla[e, e^{\prime}] = (\nabla T + 2 \mathrm{Alt} \, \iota_\rho R)(e, e^{\prime}),
\label{bcurv}
\eeqa
for $e, e^{\prime} \in \Gamma(A)$.

\subsection{Momentum sections and Hamiltonian Lie algebroids}
In this section, we explain momentum sections and Hamiltonian Lie algebroids \cite{Blohmann:2018, Blohmann:2023}, which generalize the concept of a momentum map on a symplectic manifold. For previous or related works and physical applications, see also \cite{Kotov:2016lpx, Ikeda:2019pef, Ikeda:2021fjk}.

A closed $2$-form $\omega \in \Omega^2(M)$ on a manifold $M$ is called a pre-symplectic form. A pair $(M, \omega)$, where $M$ is a manifold and $\omega$ is a pre-symplectic form, is called a pre-symplectic manifold. If $\omega$ is nondegenerate, then $(M, \omega)$ is a symplectic manifold.

A momentum section and a Hamiltonian Lie algebroid over a pre-symplectic manifold is defined as follows.
\begin{definition}\label{momsymp}[Hamiltonian Lie algebroids over pre-symplectic manifolds]
Let $(M, \omega)$ be a pre-symplectic manifold, and 
take let $(A, [-,-], \rhoa)$ be a Lie algebroid over $M$.
%

\noindent
(S1) The Lie algebroid $A$ is called 
\textit{pre-symplectically anchored} if $\omega$ satisfies
\begin{eqnarray}
&& 
{}^A \nabla \omega =0.
\label{HH1}
\end{eqnarray}
(S2) A section $\mu \in \Gamma(A^*)$ is a \textit{$\nabla$-momentum section} if it satisfies
\begin{eqnarray}
&& 
(\nabla \mu)(e) = - \iota_{\rhoa(e)} \omega,
\label{HH2}
\end{eqnarray}
for all $e \in \Gamma(A)$.

\noindent
(S3) The section $\mu$ is \textit{bracket-compatible} if it satisfies
\begin{eqnarray}
&& 
({}^A \rd \mu)(e_1, e_2) = \omega(\rhoa(e_1), \rhoa(e_2)),
\label{HH3}
\end{eqnarray}
where all $e_1, e_2 \in \Gamma(A)$.

A Lie algebroid $A$ over a pre-symplectic manifold equipped with a connection $\nabla$ and a section $\mu \in \Gamma(A^*)$ is called \textbf{Hamiltonian}
if it satisfies Eqs.~\eqref{HH1}, \eqref{HH2} and \eqref{HH3}.
\end{definition}
Suppose that $M$ is equipped with an action of a Lie group $G$,
and let $\omega$ be a symplectic form on $M$.
For the Lie algebra $\mathfrak{g}$ of $G$, the trivial bundle 
$A = M \times \mathfrak{g}$ has the structure of an action Lie algebroid 
in Example \ref{actionLA}.
We consider the trivial connection $\nabla=\rd$ on $M \times \mathfrak{g}$.
The section $\mu \in \Gamma(M \times \mathfrak{g}^*)$ is regarded as a map
$\mu: M \rightarrow  \mathfrak{g}^*$.
Then, three conditions of Definition \ref{momsymp} simplify 
as follows:
Equation \eqref{HH2} is
\begin{eqnarray}
(\rd \mu)(e) = - \iota_{\rhoa(e)} \omega.
\label{HMM2}
\end{eqnarray}
Equation \eqref{HH1} is
${}^A \nabla_e \omega = \calL_{\rhoa(e)} \omega = 0$, which follows from the definition of the $A$-connection and Eq.~\eqref{HMM2}.
Under Equations
\eqref{HH2} and \eqref{HH1}, this condition is equivalent to 
$\mu([e_1, e_2]) = \rhoa(e_1) \mu(e_2)$, which implies that 
$\mu$ is inifintesimally equivariant.
Thus, these three conditions ensure that
$\mu$ defines a momentum map on the symplectic manifold $M$.

A Hamiltonian Lie algebroid over a Poisson manifold is defined as follows \cite{Blohmann:2023}.
\begin{definition}\label{momPois}[Hamiltonian Lie algebroids over Poisson manifolds]
Let $(M, \pi)$ be a Poisson manifold with a Poisson bivector field $\pi$,
and let $(A, [-,-], \rhoa)$ be a Lie algebroid over $M$.
Similar to the pre-symplectic case, we define the following three conditions:

\noindent
(P1) The Lie algebroid $A$ is called \textit{Poisson anchored} 
if $\pi$ satisfies
\begin{eqnarray}
{}^A \nabla \pi &=& 0.
\label{PMS1}
\end{eqnarray}
(P2) A section $\mu \in \Gamma(A^*)$ is a \textit{$\nabla$-momentum section} if it satisfies
\begin{eqnarray}
\rhoa(e) &=& - \iota_{\nabla \mu(e)} \pi,
\label{PMS2}
\end{eqnarray}
for all $e \in \Gamma(A)$.

\noindent
(P3) The section $\mu$ is called \textit{bracket-compatible} if it satisfies
\begin{eqnarray}
({}^A \rd \mu) (e_1, e_2) &=& - \pi((\nabla \mu)(e_1), (\nabla \mu)(e_2)),
\label{PMS3}
\end{eqnarray}
for all $e_1, e_2 \in \Gamma(A)$.

A Lie algebroid $A$ over a Poisson manifold, equipped with a connection $\nabla$ and a section $\mu \in \Gamma(A^*)$, is called \textbf{Hamiltonian}
if it satisfies Eqs.~\eqref{PMS1}, \eqref{PMS2} and \eqref{PMS3}.
\end{definition}

If the Poisson structure $\pi$ is nondegenerate, then the corresponding symplectic form is given by $\omega = \pi^{-1}$.
In this case, a Hamiltonian Lie algebroid over a Poisson manifold 
is a Hamiltonian Lie algebroid over a symplectic manifold.

\section{Hamiltonian Lie algebroids over Dirac structures
}\label{sec:Dirac}
In this section, we define a momentum section and a Hamiltonian Lie algebroid over a Dirac structure.
\subsection{Courant algebroids and Dirac structures}
In this subsection, we review Courant algebroids and Dirac structures.

\begin{definition}\label{courantdefinition}
A Courant algebroid is a vector bundle $E$ over $M$,
equipped with a nondegenerate symmetric bilinear form
$\bracket{-}{-}$, a bilinear operation $\courant{-}{-}$ on $\Gamma(E)$,
and a bundle map called an anchor map,
$\rhoe: E \longrightarrow TM$, satisfying the following properties:
%
\begin{eqnarray}
&& 1, \quad \courant{e_1}{\courant{e_2}{e_3}} = \courant{\courant{e_1}{e_2}}{e_3} + \courant{e_2}{\courant{e_1}{e_3}}, 
  \label{courantdef1}
\\
&& 2, \quad \rhoe(\courant{e_1}{e_2}) = [\rhoe(e_1), \rhoe(e_2)], 
  \label{courantdef2}
\\
&& 3, \quad \courant{e_1}{f e_2} = f \courant{e_1}{e_2}
+ (\rhoe(e_1)f)e_2, 
  \label{courantdef3}
 \\
&& 4, \quad \courant{e}{e} = \frac{1}{2} {\cal D} \bracket{e}{e},
  \label{courantdef4}
\\ 
&& 5, \quad \rhoe(e_1) \bracket{e_2}{e_3}
= \bracket{\courant{e_1}{e_2}}{e_3} + \bracket{e_2}{\courant{e_1}{e_3}},
  \label{courantdef5}
\end{eqnarray}
where 
$e, e_1, e_2, e_3 \in \Gamma(E)$, $f \in C^{\infty}(M)$ and 
${\cal D}$ is a map from $C^{\infty}(M)$ to $\Gamma (E)$, 
defined as 
$\bracket{{\cal D}f}{e} = \rhoe(e) f$ \cite{LWX}.
\end{definition}
The bracket $\courant{-}{-}$ is called the Dorfman bracket.
A Courant algebroid is encoded in the quadruple $(E, \bracket{-}{-}, \courant{-}{-}, \rhoe)$.

\begin{example}
\label{standardCA}
The \textsl{standard Courant algebroid} is the Courant algebroid 
with $E = TM \oplus T^*M$. 
The anchor map is the natural projection $\rhott:TM\oplus T^*M \longrightarrow TM$. 
The inner product, anchor map and Dorfman bracket
are defined as follows:
\begin{align}
	\bracket{X + \alpha}{Y + \beta} &= \iota_X \beta + \iota_Y \alpha, \notag \\
\rhott(X+\alpha) &= X, \notag \\
\courant{X + \alpha}{Y + \beta} &= [X, Y] + \calL_X \beta - \iota_Y \rd \alpha,
\notag
\end{align}
for sections $X + \alpha, Y + \beta \in \Gamma(TM \oplus T^*M)$, 
where $X,Y$ are vector fields and $\alpha,\beta$ are $1$-forms. 
\end{example}
\begin{example}\label{standardCAH}
The \textsl{standard Courant algebroid with H-flux} is the Courant 
algebroid which contains the same inner product $\bracket{-}{-}$ and
anchor map $\rho:TM  \oplus T^*M \longrightarrow TM$ as 
the standard Courant algebroid. 
The Dorfman bracket is deformed by a closed $3$-form $H$ 
to the bracket defined by
\begin{eqnarray}
[{X + \alpha},{Y + \beta}]_H = [X, Y] + \calL_X \beta - \iota_Y \rd \alpha
+ \iota_X \iota_Y H. 
\label{standardDorfmanbracketH}
\end{eqnarray} 
\end{example}

\begin{definition}
A \textsl{Dirac structure} $L$ is a 
maximally isotropic subbundle of a Courant algebroid $E$, 
whose sections are closed under the Dorfman bracket.
A Dirac structure satisfies
\begin{eqnarray}
&& \bracket{e_1}{e_2}=0, \ \ (\mbox{isotropic})
\label{Diracstr1}
\\
&& \courant{e_1}{e_2} \in \Gamma(L), \ \ (\mbox{closed})
\label{Diracstr2}
\end{eqnarray}
for all $e_1, e_2 \in \Gamma(L)$.
\end{definition}

\begin{example}[Dirac structure from pre-symplectic structure]
\label{symplecticdirac}
Let $(M, \omega)$ be a pre-symplectic manifold and
$(TM \oplus T^*M, \bracket{-}{-}, \courant{-}{-}, \rhott)$ be a 
standard Courant algebroid with $H=0$ in Example \ref{standardCA}.
Define a bundle map $\omega^{\flat}: TM \rightarrow T^*M$ 
by $\omega^{\flat}(X)(Y) = \omega(X, Y)$ for all $X, Y \in \mathfrak{X}(M)$.
Then, the subbundle 
\begin{align}
L_{\omega} &= \{ X + \omega^{\flat}(X) | X \in \mathfrak{X}(M) \},
\label{sympdirac}
\end{align}
is a Dirac structure.
Indeed, since $\omega(X, Y) = - \omega(Y, X)$,
we have $\bracket{X + \omega^{\flat}(X)}{Y + \omega^{\flat}(Y)}=0$.
Moreover,
\begin{eqnarray}
\courant{X + \omega^{\flat}(X)}{Y + \omega^{\flat}(Y)} 
&=& [X, Y] + \omega^{\flat}([X, Y]),
\end{eqnarray}
which shows that $L_{\omega}$ is involutive, confirming that it forms a Dirac structure.
\end{example}

\begin{example}[Dirac structure from Poisson manifold]
\label{Poissondirac}
Let $(M, \pi)$ be a Poisson manifold, and
let $(TM \oplus T^*M, \bracket{-}{-}, \courant{-}{-}, \rhott)$ be a 
standard Courant algebroid with $H=0$.
Consider the bundle map $\pi^{\sharp}: T^*M \rightarrow TM$,
and define the subbundle $L_{\pi} \subset TM \oplus T^*M$ by
\begin{align}
L_{\pi} &:= \{ \pi^{\sharp}(\alpha) + \alpha | \alpha \in \Omega^1(M) \}.
\label{Poisdirac}
\end{align}
Then $L_{\pi}$ is a Dirac structure.
Indeed, the inner product satisfies
$\bracket{\pi^{\sharp}(\alpha) + \alpha}{\pi^{\sharp}(\beta) + \beta}=0$
since $\pi(\alpha, \beta) = - \pi(\beta, \alpha)$.
Furthermore, a direct calculation gives
\begin{eqnarray}
\courant{\pi^{\sharp}(\alpha) + \alpha}{\pi^{\sharp}(\beta) + \beta}
&=& \pi^{\sharp}([\alpha, \beta]_{\pi}) + [\alpha, \beta]_{\pi},
\end{eqnarray}
which shows that $L_{\pi}$ is involutive.
Thus $L_{\pi}$ is a Dirac structure.
\end{example}

\subsection{Momentum sections and Hamiltonian Lie algebroids over Dirac structures}\label{sec:MSDirac}

We consider a Courant algebroid on $TM \oplus T^*M$.
Note that the Courant algebroid is not necessarily assumed to 
the standard Courant algebroid. For instance, 
we can take a Courant algebroid with a $3$-form $H$ in Example \ref{standardCAH}.

Let $(TM \oplus T^*M, \bracket{-}{-}, \courant{-}{-}, \rhott)$ 
be a Courant algebroid and let $L \subset TM \oplus T^*M$ be a Dirac structure.
Furthermore, let $(A, [-,-], \rhoa)$ be a Lie algebroid over $M$, and
let $\nabla$ be a connection on $A$.
Let $\mu$ be a section $\mu \in \Gamma(A^*)$.
Then $\rhoa + \nabla \mu$ defines a bundle map $\rhoa + \nabla \mu: A \rightarrow  TM \oplus T^*M$. Alternatively, it can be considered as a section $\rhoa + \nabla \mu 
\in \Gamma((TM \oplus T^*M) \otimes A^*)$.

We define momentum sections and Hamiltonian Lie algebroids over Dirac structures.
\begin{definition}\label{momsecdirac}
\begin{itemize}
   \item[(D1)] Let $v + \gamma \in \Gamma((TM \oplus T^*M) \otimes A^*)$,
where $v \in \Gamma(TM \otimes A^*)$ and $\gamma \in \Gamma(T^*M \otimes A^*)$.
Suppose that $(v + \gamma)(e) \in \Gamma(L)$ for every $e \in \Gamma(A)$.
The Lie algebroid $A$ is called \textit{Dirac anchored} 
if $({}^A \nabla (v + \gamma))(e_1, e_2) \in \Gamma(L)$ holds for all $e_1, e_2 \in \Gamma(A)$.
   \item[(D2)] A section $\mu \in \Gamma(A^*)$ is called a $\nabla$-\textit{momentum section} if
$(\rhoa + \nabla \mu)(e)$ is an element of the Dirac structure $L$ 
for every $e \in \Gamma(A)$,
and the map $\rhoa + \nabla \mu$ is a Lie algebroid morphism from $A$ to $L$.
   \item[(D3)] The section $\mu$ is called \textit{bracket-compatible} if it satisfies
\begin{eqnarray}
({}^A \rd \mu) (e_1, e_2) = \frac{1}{2} \bracketm{(\rhoa + \nabla \mu)(e_1)}{(\rhoa + \nabla \mu)(e_2)},
\label{D03}
\end{eqnarray}
for every $e_1, e_2 \in \Gamma(A)$.
Here, the bilinear form is given by $\bracketm{X + \alpha}{Y + \beta} = \alpha(Y) - \beta(X)$ for $X, Y \in \mathfrak{X}(M)$ and $\alpha, \beta \in \Omega^1(M)$.
\end{itemize}

A Lie algebroid $A$ with a connection $\nabla$ and a section $\mu \in \Gamma(A^*)$ is called \textbf{Hamiltonian}
if the condition (D1), (D2)  and (D3) are satisfied.
\end{definition}

\begin{remark}
In condition (D1),
$v + \gamma$ is regarded as a bundle map $v + \gamma:A \rightarrow TM \oplus T^*M$, and ${}^A \nabla (v + \gamma)$ can be regarded as a bundle map,
${}^A \nabla (v + \gamma): A \times A \rightarrow TM \oplus T^*M$.
Thus, the condition (D1) is equivalent to the following statement:

\textit{For any bundle map $v + \gamma: A \rightarrow TM \oplus T^*M$ such that $\mathrm{Im}(v + \gamma) \subset L$, it holds that $\mathrm{Im}({}^A \nabla (v + \gamma)) \subset L$.}
\end{remark}

\begin{remark}
A Lie algebroid $A$ can be regarded as a Dirac structure of $A \oplus A^*$. Thus, $\rhoa + \nabla \mu: A \rightarrow TM \oplus T^*M$ is a map between two Dirac structures. A \textit{Dirac map} between two Dirac structures is defined in 
\cite{BursztynCrainic1, BursztynCrainic2}.
\end{remark}

Definition \ref{momsecdirac} includes Hamiltonian Lie algebroids over pre-symplectic and Poisson manifolds as special cases. 
\begin{proposition}
Let $(M, \omega)$ be a pre-sympletic manifold and
let $(A, [-,-], \rhoa)$ be a Lie algebroid over $M$ with a connection $\nabla$.
Suppose that $\mu \in \Gamma(A^*)$ is a momentum section on a pre-symplectic manifold, i.e., it satisfies 
Equations \eqref{HH1}--\eqref{HH3}.
Then, 
$A$ is a Hamiltonian Lie algebroid over the Dirac structure $L_{\omega}$ with 
$H=0$ and $\mu$ is a momentum section.
\end{proposition}

\begin{proof}
Suppose that Equations \eqref{HH1}--\eqref{HH3} hold.

Since $(v + \gamma)(e)$ is an element of the Dirac structure $L_{\omega}$,
we have $\omega^{\flat} (v) = \gamma$ from Equation \eqref{sympdirac}.
Applying ${}^A \nabla$, we obtain ${}^A \nabla (\omega^{\flat} (v)) = {}^A \nabla \gamma$.
Using Eq.~\eqref{HH1}, we rewrite this as
$\omega^{\flat}  ({}^A \nabla (v)) = {}^A \nabla \gamma$.
This implies that 
$({}^A \nabla(v + \gamma))(e_1, e_2) \in L_{\omega}$. 
Thus, the condition (D1) holds.

From Equation \eqref{HH2} and \eqref{sympdirac},
we obtain 
\begin{align}
\omega^{\flat} (\rho) = \nabla \mu.
\label{symprhomu}
\end{align}
Using Eq.~\eqref{symprhomu} and $\rd \omega =0$, we compute
\begin{align}
& \bracket{(\rhoa + \nabla \mu)(e_1)}{(\rhoa + \nabla \mu)(e_2)}
= 0,
\\
& \courant{(\rhoa + \nabla \mu)(e_1)}{(\rhoa + \nabla \mu)(e_2)}
= (\rhoa + \nabla \mu)([e_1, e_2]),
\end{align}
for $e_1, e_2 \in \Gamma(A)$. 
Thus, $(\rhoa + \nabla \mu)(e) \in L_{\omega}$,
which establishes condition (D2).

A direct calculation using Eq.~\eqref{symprhomu} gives
\begin{align}
\frac{1}{2} \bracketm{(\rhoa + \nabla \mu)(e_1)}{(\rhoa + \nabla \mu)(e_2)}
= \omega(\rho(e_1), \rho(e_2)).
\end{align}
Thus, Eq.~\eqref{HH3} ensures that condition (D3) holds.
\hfill\qed
\end{proof}

\begin{proposition}
Let $(M, \pi)$ be a Poisson manifold and
let $(A, [-,-], \rhoa)$ be a Lie algebroid over $M$ with a connection $\nabla$.
Suppose that $A$ is Hamiltonian, i.e., it satisfies Eqs.~\eqref{PMS1}--\eqref{PMS3}.
Then, 
$A$ is a Hamiltonian Lie algebroid over the Dirac structure $L_{\pi}$ 
with $H=0$, and 
$\mu$ is a momentum section if $\bracket{{}^AS}{\mu}=0$,
where ${}^AS$ is the basic curvature of $A$.
\end{proposition}

\begin{proof}
Suppose that Eqs.~\eqref{PMS1}--\eqref{PMS3} hold.

Suppose $(v + \gamma)(e)$ is an element of the Dirac structure $L_{\pi}$,
we have $v = \pi^{\sharp} (\gamma)$ from Eq.~\eqref{Poisdirac}.
Applying ${}^A \nabla$, 
we obtain ${}^A \nabla v = {}^A \nabla (\pi^{\sharp} (\gamma))$.
Using Eq.~\eqref{PMS1}, we rewrite this as
${}^A \nabla v = \pi^{\sharp} ({}^A \nabla (\gamma))$,
This implies that
$({}^A \nabla(v + \gamma))(e_1, e_2) \in \calL_{\pi}$
Thus, condition (D1) holds.

From Eqs.~\eqref{PMS2} and \eqref{Poisdirac},
we obtain 
\begin{align}
\rho = \pi^{\sharp} (\nabla \mu).
\label{Poisrhomu}
\end{align}
Using Eq.~\eqref{Poisrhomu} and the condition $[\pi, \pi]_S=0$, 
we compute
\begin{align}
& \bracket{(\rhoa + \nabla \mu)(e_1)}{(\rhoa + \nabla \mu)(e_2)}
= 0,
\\
& \courant{(\rhoa + \nabla \mu)(e_1)}{(\rhoa + \nabla \mu)(e_2)}
= (\rhoa + \nabla \mu)([e_1, e_2]),
\end{align}
for all $e_1, e_2 \in \Gamma(A)$. Thus, $(\rhoa + \nabla \mu)(e)$ is an element of $L_{\pi}$, which establishes condition (D2) \cite{HirotaIkeda2025}.

Substituting \eqref{Poisrhomu}, we obtain
\begin{align}
\frac{1}{2} \bracketm{(\rhoa + \nabla \mu)(e_1)}{(\rhoa + \nabla \mu)(e_2)}
= - \pi((\nabla \mu)(e_1), (\nabla \mu)(e_2)).
\end{align}
Thus, Eq.~\eqref{PMS3} ensures that condition (D3) holds.
\hfill\qed
\end{proof}

\subsection{Formulas and properties}
In this section, we prove some properties and useful formulas 
of a Hamiltonian Lie algebroid over Dirac structure.
\begin{lemma}
Under condition (D2), Eq.~\eqref{D03} is equivalent to
\begin{eqnarray}
\calL_{\rho(e_1)} \mu(e_2) = \mu([e_1, e_2]).
\label{D04}
\end{eqnarray}
\end{lemma}
\begin{proof}
The left-hand side of Eq.~\eqref{D03} is
\begin{eqnarray}
({}^A \rd \mu) (e_1, e_2) = 
\calL_{\rho(e_1)} \mu(e_2) - \calL_{\rho(e_2)} \mu(e_1) - \mu([e_1, e_2]).
\end{eqnarray}
Since $(\rhoa + \nabla \mu)(e)$ is an element of the Dirac structure
from condition (D2), we have 
$\bracket{(\rhoa + \nabla \mu)(e_1)}{(\rhoa + \nabla \mu)(e_2)} =0$. 
Using this condition, the right-hand side of 
Eq.~\eqref{D03} reduces to $\calL_{\rho(e_2)}\mu(e_1)$.
Substituting these into both sides of Eq.~\eqref{D03}, we obtain Eq.~\eqref{D04}.
\hfill\qed
\end{proof}
Eq.~\eqref{D04} implies that $\mu$ is infinitesimally equivariant.

\begin{lemma}\label{Dorfmaninvolutive}
Condition (D2) implies
\begin{eqnarray}
[\rhoa(e_1) + (\nabla \mu)(e_1), \rhoa(e_2) + (\nabla \mu)(e_2)]_D
&=& (\rhoa + \nabla \mu)([e_1, e_2]).
\end{eqnarray}
\end{lemma}
\begin{proof}
Since $\rhoa$ is the anchor map of the Lie algebroid $A$, it
satisfies $[\rhoa(e_1), \rhoa(e_2)] = \rhoa([e_1, e_2])$.
From condition (D2), 
$[\rhoa(e_1) + (\nabla \mu)(e_1), \rhoa(e_2) + (\nabla \mu)(e_2)]_D$
is involutive. 
Thus, there exists $e^{\prime} \in \Gamma(A)$
such that
$(\rhoa + \nabla \mu)(e^{\prime})$

From the equation of the $\rhoa$ part, we conclude that 
$e^{\prime} = [e_1, e_2]$.
\hfill\qed
\end{proof}

\begin{proposition}
Let $TM \oplus T^*M$ be the standard Courant algebroid with $H$. Then, condition (D2) is equivalent to
\begin{eqnarray}
&& \bracket{\rhoa(e_1)}{\nabla \mu (e_2)}+  \bracket{\rhoa(e_2)}{\nabla \mu (e_1)} = 0,
\label{equivd21}
\\
&& ({}^A \nabla_{e_1} \nabla_X \mu)(e_2) + R(X, \rhoa(e_2), e_1, \mu)
- H(\rhoa(e_1), \rhoa(e_2),-) =0,
\label{equivd22}
\end{eqnarray}
for all $X \in \mathfrak{X}(M)$ and $e_1, e_2 \in \Gamma(A)$,
where $R \in \Omega^2 (M, A \otimes A^*)$ is the curvature.
\end{proposition}

\begin{proof}
Eq.~\eqref{equivd21} corresponds to condition,
$\bracket{\rhoa(e_1) + \nabla \mu (e_1)}{\rhoa(e_2) + \nabla \mu (e_2)} = 0$.

For condition \eqref{equivd22},
we compute the Dorfman bracket of the standard Courant algebroid with $H$:
$[\rhoa(e_1) + \nabla \mu (e_1), \rhoa(e_2) + \nabla \mu (e_2)]_D$.
Using Lemma \ref{Dorfmaninvolutive},
we obtain
\begin{eqnarray}
[\rhoa(e_1) + (\nabla \mu)(e_1), \rhoa(e_2) + (\nabla \mu)(e_2)]_D
&=& \rhoa([e_1, e_2]) + (\nabla \mu)([e_1, e_2]).
\label{inv11}
\end{eqnarray}
Equation \eqref{inv11} is equal to
\begin{eqnarray}
\calL_{\rho(e_1)} (\nabla \mu)(e_2) - \iota_{\rho(e_2)} \rd (\nabla \mu)(e_1)
+ \iota_{\rho(e_1)} \iota_{\rho(e_2)} H 
&=& (\nabla \mu)([e_1, e_2]).
\end{eqnarray}
Rewriting this equation using the connection and the curvature $R$, 
we obtain condition \eqref{equivd22}. 
\hfill\qed
\end{proof}

\begin{proposition}
Condition (D3), i.e., Eq.~\eqref{D03}, is equivalent to
\begin{eqnarray}
\bracket{\mu}{T(e_1, e_2)} = 
- \frac{1}{2} \bracketm{(\rhoa + \nabla \mu)(e_1)}{(\rhoa + \nabla \mu)(e_2)},
\label{D3torsion}
\end{eqnarray}
for all $e_1, e_2 \in \Gamma(A)$.
\end{proposition}

\begin{proof}
From the definition of the $A$-differential \eqref{LAdifferential}
and the $A$-torsion \eqref{Etorsion},
the left-hand side of condition \eqref{D03} is 
\begin{eqnarray}
{}^A \rd \mu(e_1, e_2) &=& \bracket{\rhoa(e_1)}{(\nabla \mu)(e_2)} - 
\bracket{\rhoa(e_2)}{(\nabla \mu)(e_1)} + \bracket{\mu}{T(e_1, e_2)}
\nonumber \\
&=& \bracketm{(\rhoa + (\nabla \mu))(e_1)}{(\rhoa + (\nabla \mu))(e_2)}
+ \bracket{\mu}{T(e_1, e_2)}.
\end{eqnarray}
Substituting this into Eq.~\eqref{D03}, we obtain Eq.~\eqref{D3torsion}.
\hfill\qed
\end{proof}

In general, $TM \oplus T^*M$ is decomposed by a Lie bialgebroid $(L, L^*)$ as 
$L \oplus L^* = TM \oplus T^*M$.
For $L$, the maximal isotropic subbundle is given by
\begin{eqnarray}
L_{\varepsilon} &=& \{ s + \varepsilon s| s \in L \},
\end{eqnarray}
for some $\varepsilon \in \wedge^2 L^*$.
Here, $\varepsilon$ is regarded as a map 
$\varepsilon: L \rightarrow L^*$.
$L_{\varepsilon}$ is a Dirac structure 
if and only if $\varepsilon$ satisfies
\begin{eqnarray}
\rd_L \varepsilon + \frac{1}{2} [\varepsilon, \varepsilon] =0,
\label{epsicondi}
\end{eqnarray}
where $[-,-]$ is a Lie bracket on $L^*$,
and $\rd_L: \Gamma(\wedge^m L^*) \rightarrow \Gamma(\wedge^{m+1} L^*) $ is 
induced from the decomposition of the operator,
$\calD = \rd_L + \rd_{L^*}$, where
$\rd_L: C^{\infty}(M) \rightarrow \Gamma(L^*)$
and $\rd_{L^*}: C^{\infty}(M) \rightarrow \Gamma(L)$.
For details, see \cite{LWX}.

For $(TM \oplus T^*M) \otimes A^*$, we obtain the decomposition
$(L \oplus L^*) \otimes A^*$. 
From the above, we have the following proposition.
\begin{proposition}
The Dirac structure is given by
\begin{eqnarray}
L_{\varepsilon} &=& \{ (\sigma + \varepsilon \sigma)(e) | \sigma \in L \otimes A^*, e \in \Gamma(A) \},
\end{eqnarray}
i.e., $(\sigma + \varepsilon \sigma)(e)$ is an element of the Dirac structure $L_{\varepsilon}$ for all $e \in \Gamma(A)$,
where $\varepsilon \in \wedge^2 L^*$ satisfies condition \eqref{epsicondi}.
\end{proposition}
Then, condition (D1) is
${}^A \nabla (\sigma + \varepsilon \sigma)(e_1, e_2) \in L_{\varepsilon}$,
which is equivalent to ${}^A \nabla \varepsilon = 0$.

\section{Examples}\label{sec:Exam}
In this section, we present some examples. 

\begin{example}[Bundles of Lie algebras]
We present an almost trivial example. 
A bundle of a Lie algebra forms a Lie algebroid with a zero anchor $\rhoa=0$.
In this case, the $A$-connection is also zero, ${}^A \nabla =0$,
for any given connection $\nabla$. 
Thus, condition (D1) is always satisfied.

Since (D2) becomes $(0 + \nabla \mu)(e) = (\nabla \mu)(e) \in \Gamma(L)$,
this implies that $(\nabla \mu)$ is horizontal to the direction of $L$ on $TM \oplus T^*M$.

Since $\rhoa=0$, condition (D3) reduces to
$({}^A \rd \mu)(e_1, e_2) = 0$ for all $e_1, e_2 \in \Gamma(A)$.
On the other hand, since $({}^A \rd \mu)(e_1, e_2) = - \bracket{\mu}{[e_1, e_2]}$, condition (D3) is equivalent to $\bracket{\mu}{[e_1, e_2]} =0$, which is the same as the condition of a Hamiltonian Lie algebroid for bundles of Lie algebras over pre-symplectic or Poisson manifolds.
For instance, $\mu=0$ is obviously a trivial solution satisfying conditions (D1)-(D3).
\end{example}

\begin{example}[Hamiltonian Lie algebroids 
over Dirac structures twisted by $H$]
We consider the standard Courant algebroid from Example \ref{standardCA},
where the closed $3$-form $H$ is not necessarily zero.
If we consider the following subbundle:
\begin{align}
L_{\omega} &= \{ X + \omega^{\flat}(X) | X \in \mathfrak{X}(M) \},
\label{sympHdirac}
\end{align}
for some $2$-form $\omega$,
then $L_{\omega}$ is a Dirac structure if and only if $\rd \omega + H = 0$.

Now let $A$ be a Lie algebroid over $M$ with a connection $\nabla$,
By imposing conditions (D1), (D2) and (D3)
for a section $\mu \in \Gamma(A^*)$,
we obtain a Hamiltonian Lie algebroid on 
the Dirac structure $L_{\omega}$ twisted by $H$, satisfying 
$\rd \omega + H = 0$.
The explicit conditions remains formally the same as in the pre-symplectic case:
\begin{eqnarray}
&& 
{}^A \nabla \omega =0.
\label{tHH1}
\\ && 
(\nabla \mu)(e) = - \iota_{\rho(e)} \omega,
\label{tHH2}
\\ && 
({}^A \rd \mu)(e_1, e_2) = \omega(\rho(e_1), \rho(e_2)),
\label{tHH3}
\end{eqnarray}
except for $\omega$ is no longer closed but satisfies
$\rd \omega + H = 0$.
\end{example}

\begin{example}[Hamiltonian Lie algebroids over twisted Poisson structures]
We consider a standard Courant algebroid with a $3$-form $H$.
Let $\pi^{\sharp}: T^*M \rightarrow TM$ be the associated bundle map,
and define the graph of $\pi^{\sharp}$ as
\begin{align}
L_{\pi} &:= \{ \pi^{\sharp}(\alpha) + \alpha | \alpha \in \Omega^1(M) \}.
\end{align}
THe, $L_{\pi}$ is a Dirac structure 
if and only if $\pi$ is the twisted Poisson structure as defined 
in Example \ref{tPoisson}.

Now, let $A$ be a Lie algebroid over $M$ and let $\mu \in \Gamma(A^*)$.
In this setting, we consider a Hamiltonian Lie algebroid over 
a Dirac structure $L_{\pi}$. We take an element 
$\rhoa + \nabla \mu \in \Gamma((TM \oplus T^*M) \otimes A^*)$ 
and impose conditions (D1), (D2) and (D3).
Under three condition, the pair $(A, \mu)$ defines a Hamiltonian Lie algebroid 
if $\pi$ and $\mu$ satisfies the following three conditions,
\begin{eqnarray}
{}^A \nabla \pi &=& 0,
\label{tPMS1}
\\
\rho(e) &=& - \iota_{\nabla \mu(e)} \pi,
\label{tPMS2}
\\
({}^A \rd \mu) (e_1, e_2) &=& - \pi(\nabla \mu(e_1), \nabla \mu(e_2)).
\label{tPMS3}
\end{eqnarray}
for every $e, e_1, e_2 \in \Gamma(A)$.
Eqs.~\eqref{tPMS1}--\eqref{tPMS3} are formally identical to those in the Poisson manifold case. However, here the bivector field $\pi$ is not a Poisson structure, but a twisted Poisson structure, which satisfies.
\end{example}

\begin{example}[Hamiltonian Dirac structures over themselves]
Any Dirac structure $L \subset TM \oplus T^*M$ is a Lie algebroid.
Thus, we can take $A=L$.

The condition (D1) in Definition \ref{momsecdirac} 
is satisfied if we choose a connection $\nabla$ over $L$ as a vector bundle
and the corresponding $L$-connection ${}^L \nabla$.

Now, let $\mu \in \Gamma(L^*)$. The condition (D2) states that
the map $\rho_L + \nabla \mu: L \rightarrow TM \oplus T^*M$ is the (forward) 
Dirac map $\rho_L + \nabla \mu: L \rightarrow L$ defined in
\cite{BursztynRadko}.

The condition (D3) is given by 
\begin{eqnarray}
({}^L \rd \mu) (e_1, e_2) = \frac{1}{2} \bracketm{(\rho_L + \nabla \mu)(e_1)}{(\rho_L + \nabla \mu)(e_2)},
\label{DD03}
\end{eqnarray}
for all $e_1, e_2 \in \Gamma(L)$.
\end{example}

\begin{example}[Hamiltonian Lie algebroids over involutive foliations]
Let $\calF \subset TM$ be a regular foliation. 
Define its annihilator as
\begin{eqnarray}
\mathrm{Ann}(\calF) = \{\alpha \in T^*M\, |\, \alpha|_{\calF_x} = 0 \ \mbox{for} \ x \in M \}.
\label{ann}
\end{eqnarray}
Then, the subbundle
$L_{\calF} :=  \calF \oplus \mathrm{Ann}(\calF)$
form a Dirac structure. 

Now, suppose that a Lie group $G$ acts on $M$ while preserving $\calF$.
Then, the induced map $\rho$ from its Lie algebra $\mathfrak{g}$ to $\calF$ satisfies $\mathrm{Im}(\rho) \subset \calF$ and is involutive.
More generally, for any Lie algebroid $A$ with an anchor map $\rho$, the image $\calF = \mathrm{Im}(\rho)$ defines an involutive (possibly singular) foliation \cite{LGLR2024},
and $L_{\calF}$ forms a Dirac structure.

In this setting, if 
a connection $\nabla$ on $A$ and a section $\mu \in \Gamma(A^*)$ 
satisfy conditions (D1), (D2) and (D3) in Definition \ref{momsecdirac},
$(L_F, A, \nabla, \mu)$ is a Hamiltonian Lie algebroid.

We consider a simple example.

Consider $M= \bR^3$ with coordinates $(x, y, z)$.
If we take $\calF = \spans{\partial_x}$, then 
$\mathrm{Ann}(\calF) = \spans{\rd y, \rd z}$.
Then a Dirac structure is 
$L_{x} = L_{\calF} := \spans{\partial_x, \rd y, \rd z}$.
Take $A := \spans{\partial_x} \subset TM$, which is trivially 
a Lie algebroid with $\rhoa = \mathrm{id}$.
We can take a trivial connection $\nabla = \rd$. 
Then, ${}^A \nabla = \rd x \partial_x$.

Since $A \subset L_{x}$, the condition (D1) is trivially satisfied.
Therefore this Lie algebroid is Dirac anchored.

Consider a $1$-form $\mu = f(x, y, z) \rd x \in \Gamma(A^*)$,
where $f(x, y, z)$ is a coefficient function.
$\rhoa + \nabla \mu$ is calculated as
\begin{eqnarray}
\rhoa + \nabla \mu = \mathrm{id} + \partial_y f \rd y \wedge \rd x 
+ \partial_z f \rd z \wedge \rd x 
\end{eqnarray}
For every $e = g(x, y, z) \partial_x \in \Gamma(A)$,
the map $\rhoa + \nabla \mu$ is a Lie algebroid morphism since 
\begin{eqnarray}
(\rho + \nabla \mu)(e) = g \partial_x + g \partial_y f \rd y
+ g \partial_z f \rd z. 
\end{eqnarray}
Therefore the condition (D2) is satisfied.

Since ${}^A \rd \mu = 0$,
and 
$\bracketm{(\rhoa + \nabla \mu)(e_1)}{(\rhoa + \nabla \mu)(e_2)}=0$,
for every $e_1, e_2 \in \Gamma(A)$, 
(D3) is trivially satisfied.

Therefore, $A$ in the above example is a Hamiltonian Lie algebroid over the Dirac structure $L_{x}$.

\end{example}


\section{Physical models}\label{sec:physmod}
In this section, we explore an application of Hamiltonian Lie algebroids to physical models,
focusing on topological sigma models in two dimensions.

A Hamiltonian mechanics and sigma models with Hamiltonian Lie algebroid 
structures over pre-symplectic manifolds have already been discussed in the author's previous work \cite{Ikeda:2019pef}.
Here, we construct two dimensional sigma models with boundary 
with Hamiltonian Lie algebroid structures over Poisson manifolds and Dirac structures.

\subsection{Gauged Poisson sigma models (GPSM)}\label{sec:GPSM}
We begin with the Poisson sigma model \cite{Ikeda:1993aj, Ikeda:1993fh, Schaller:1994es}.
It is a sigma model from a two-dimensional manifold $\Sigma$ to a Poisson manifold $M$. More precisely, the space of fields is given by the vector bundle $\Hom(T\Sigma, T^*M)$.
In this section, we assume that $\Sigma$ is closed, $\partial \Sigma = \emptyset$.

Fundamental fields are a smooth map $X:\Sigma \rightarrow M$\footnote{In this section, notation $X$ is a map between two manifolds in a sigma model, not a vector field.},
and a $1$-form $Z$ taking a value in the pullback bundle $X^*T^*M$,
$Z \in \Omega^1(\Sigma, X^* T^*M)$.
Let $\sigma^{\mu} (\mu=1,2)$ be Local coordinates on $\Sigma$.
Let $X^i$ and $Z_i \ (i=1, \ldots, d)$, where $d=\mathrm{dim}(M)$, be local coordinates on $\Hom(T\Sigma, T^*M)$.

The action functional of the Poisson sigma model
is given by
\begin{eqnarray}
S_{\pi} &=& \int_{\Sigma} 
\left(\bracket{Z}{\rd X} + \frac{1}{2} (\pi \circ X) (Z, Z) \right)
\nonumber \\
&=& \int_{\Sigma}
\left(Z_i \wedge \rd X^i + \frac{1}{2} \pi^{ij}(X) Z_i \wedge Z_j
\right).
\label{PSM}
\end{eqnarray}
where in the second line, all fields are expressed using local coordinates of the target space $M$.

The Euler Lagrange equations derived from Eq.~\eqref{PSM} yield the constraint equations:
$\rd X^i + \pi^{ij}(X) Z_j =0$.
Precisely, only the space component with respect to $\Sigma$ is the constraint.
These constraints are the first class, meaning that they are involutive under Poisson brackets if and only if $\pi$ is a Poisson bivector field.
Associated gauge transformations are given by
\begin{eqnarray}
\delta X^i &=& - \pi^{ij}(X) t_j,
\label{gauge01}
\\
\delta Z_i &=& \rd t_i + \partial_i \pi^{jk}(X) Z_j t_k,
\label{gauge02}
\end{eqnarray}
where $\partial_i = \frac{\partial}{\partial X^i}$,
and the gauge parameter $t_i$ is a function on $\Sigma$, i.e., $t_i \in C^{\infty}(\Sigma, X^*T^*M)$.
The action functional \eqref{PSM} is gauge invariant under 
these gauge transformations 
if and only if $\pi$ is a Poisson bivector field.

Constructing the gauged Poisson sigma model, we introduce a Lie algebroid $A$ over $M$, and construct a deformation of 
the Poisson sigma model which incorporates the structure of $A$\footnote{Note that $A$ is assumed to be independent of a Poisson structure $\pi$.}.
Physically, this corresponds to 'gauging' procedure. \cite{Hull:1990ms, Zucchini:2008cg} 
To do this, we extend the space of vector bundle morphisms $\Hom(T\Sigma, A)$ to the space to $\Hom(T\Sigma, T^*M \oplus A)$.
New fields 
$W = W^a e_a \in \Omega^1(\Sigma, X^*A)$ and $Y = Y_a e^a \in \Gamma(\Sigma, X^*A^*)$ are introduced, where $e_a$ and $e^a$ are local basis of each space, respectively.

We consider the following action functional by adding the term $S_L$ 
to Eq.~\eqref{PSM}:
\begin{eqnarray}
S &=& S_{\pi} + S_L
\nonumber \\
&=& S_{\pi} + \int_{\Sigma} \left(\inner{Y}{\rd W} 
- \bracket{\rho(Y)}{Z}
+ \frac{1}{2} \inner{[W, W]}{Y}
\right).
\nonumber \\
&=& S_{\pi} + \int_{\Sigma} \left(
Y_a \wedge \rd W^a - \rho^i_a(X) Z_i \wedge W^a 
+ \frac{1}{2} C_{ab}^c(X) W^a \wedge W^b Y_c
\right),
\label{GPSM}
\end{eqnarray}
where $\rhoa = \rho$ is the anchor map of the Lie algebroid $A$
satisfying $\rho(e_a) = \rho^i_a(x) \partial_i$,
and $C_{ab}^c$ is the structure functions of 
the Lie bracket, $[e_a, e_b] = C_{ab}^c e_c$ 
for a local basis $e_a \in \Gamma(A)$.
We refer to the action functional \eqref{GPSM} 
as the \textit{gauged Poisson sigma model} (GPSM).

By introducing a connection $\nabla$ on a Lie algebroid $A$, 
we can rewrite the action functional \eqref{GPSM} in a 
manifestly covariant form.
Let $\omega_{ai}^b$ denote the connection $1$-form of $\nabla$, 
satisfying: 
$\nabla e_a = \omega_{ai}^b \rd x^i \otimes e_b$.
Then, the covariant form of Eq.~\eqref{GPSM} with respect to 
the target vector bundle over $M$ is given by
\begin{align}
S^{\nabla} &= \int_{\Sigma}
\left(Z^{\nabla}_i \wedge \rd X^i 
+ \frac{1}{2} \pi^{ij}(X) Z^{\nabla}_i \wedge Z^{\nabla}_j
+ Y_a \wedge \rdd W^a - \rho^i_a(X) Z^{\nabla}_i \wedge W^a 
- \frac{1}{2} T_{ab}^c(X) W^a \wedge W^b Y_c
\right),
\label{CGPSM}
\end{align}
where 
$Z^{\nabla}_i := Z_i - \omega_{ai}^b W^a Y_b$,
$\rdd W^a := \rd W^a - \omega^a_{bi} W^b \rd X^i$.
Here, $T_{ab}^c$ is the local coordinate expression of the pullback of 
the $A$-torsion, as defined in \eqref{Etorsion}.

\subsection{Consistency conditions of GPSM}\label{sec:consistGPSM}
In order to ensure the consistency of the action functional \eqref{CGPSM}, 
we solve geometric conditions required for the target space.
To achieve this, we can use the technique of traditional constrained mechanics and BV, BFV approach, either by analyzing the constraints in \eqref{CGPSM} to determine the conditions for them to the first class, or by studying gauge transformations and the conditions required for \eqref{CGPSM} to be gauge invariant.

Instead of employing the traditional formalism of constrained systems and gauge theories, 
we adopt alternative approach based on the AKSZ formulation \cite{Alexandrov:1995kv}. 
Formulating the GPSM using the AKSZ method provides a mathematically clearer
framework.

To reformulate the GPSM, we proceed as follows.
Given a Lie algebroid $A$, we consider the graded manifold $T^*[1]A^*$. 
We construct a QP-structure (a differential graded (dg) structure) on $T^*[1]A^*$ from the Poisson structure $\pi$ and the Lie algebroid structure on $A$. 
The QP-structure condition directly imposes compatibility conditions between the Poisson structure and the Lie algebroid structure.
A sigma model is constructed on the mapping space between two graded manifolds: $\calM= \Map(T[1]\Sigma, T^*[1]A^*)$.
According to the AKSZ procedure, a QP-manifold on the mapping space $\calM$ is induced from the QP-structure on $T^*[1]A^*$.
The QP-manifold structure on $\calM$
is nothing but the AKSZ formulation of the GPSM.

Concrete construction is as follows.
A QP-manifold of degree $n$ is a graded manifold equipped with a graded symplectic form $\gomega$ of degree $n$,
a vector field $Q$ of degree one, satisfying $Q^2=0$ and $\calL_Q \gomega = 0$.
If $n \neq 0$, there exists a Hamiltonian function $\Theta$ such that
$Q = \{\Theta, -\}$ satisfying $\{\Theta, \Theta \} = 0$.

We take $T^*[1]A^*$ for the GPSM. 
We take local coordinates on $T^*[1]A^*$:
$(x^i, z_i, y_a, w^a)$ of degree $(0, 1, 0, 1)$, where
$x^i$ are coordinates on $M$ and $z_i$ are conjugate coordinates on $T^*[1]M$,
$y_a$ are coordinates on $A^*$ and $w^a$ are conjugate coordinates on $T^*[1]A^*$.
The canonical graded symplectic form on $T^*[1]A^*$ is of degree one
and given by
\begin{eqnarray}
\gomega &=& \rd x^i \wedge \rd z_i + \rd y_a \wedge \rd w^a.
\label{gsymp}
\end{eqnarray}
The corresponding graded Poisson brackets are
\begin{eqnarray}
\{x^i, z_j \} &=& \delta^i_j,
\\
\{y_a, w^b \} &=& \delta_a^b.
\end{eqnarray}
We define the Hamiltonian functions $\Theta_{\pi}$ and  $\Theta_A$
corresponding to the Poisson structure and the Lie algebroid structure:
\begin{eqnarray}
\Theta_{\pi} &=& \frac{1}{2} \pi^{ij}(x) z_i z_j,
\\
\Theta_A &=& \rho^i_a(x) z_i w^a + \frac{1}{2} C_{ab}^c(x) w^a w^b y_c.
\end{eqnarray}
where 
$\pi = \frac{1}{2} \pi^{ij}(x) \partial_i \wedge \partial_j$ 
is the Poisson bivector field,
$\rho(e_a) = \rho^i_a(x) \partial_i$,
and $[e_a, e_b] = C_{ab}^c c_e$ for a local basis $e_a$ on $A$.
The conditions $\Theta_{\pi}$ to be homological is
$\{\Theta_{\pi}, \Theta_{\pi} \} =0$. It holds if and only if 
$\pi$ is a Poisson bivector field.
Moreover $\{\Theta_{A}, \Theta_{A} \} =0$ if and only if
$A$ is a Lie algebroid.

To make the structure manifestly covariant, 
we introduce a connection $\nabla: \Gamma(A) 
\rightarrow \Gamma(A \otimes T^*M)$ on $A$.
Let $\omega_{ai}^b \rd x^i$ be a connection $1$-form for $\nabla$.
The covariantized coordinates of $z_i$ are introduced as
$z_i^{\nabla} := z_i - \omega_{ai}^b {w}^a y_b$.
Even if we introduce the connection $\nabla$,
the graded symplectic form remains canonical, $\gomega^{\nabla} = \gomega$.
From direct calculations, nonzero graded Poisson brackets including $z_i^{\nabla}$ are 
\begin{eqnarray}
\{x^i, z_j^{\nabla} \} &=& \delta^i_j,
\\
\{y_a,  w^b \} &=& \delta_a^b,
\\
\{w^a, z_j^{\nabla} \} &=& \omega_{bi}^a w^b,
\\
\{y_a, z_j^{\nabla} \} &=& \omega_{ai}^b y_b,
\\
\{z_i^{\nabla}, z_j^{\nabla} \} &=& R_{ija}^b w^a y_b.
\end{eqnarray}
The covariantized homological function compatible with the connection 
takes following forms:
\begin{eqnarray}
\Theta_{\pi}^{\nabla} &=& \frac{1}{2} \pi^{ij}(x) z^{\nabla}_i z^{\nabla}_j,
\\
\Theta_A^{\nabla} &=& - \rho^i_a(x) z^{\nabla}_i w^a - \frac{1}{2} T_{ab}^c(x) w^a w^b y_c.
\end{eqnarray}
Note that $\Theta_A^{\nabla} = \Theta_A$ but 
$\Theta_{\pi}^{\nabla} \neq \Theta_{\pi}$.
Thus, $\{\Theta_A^{\nabla}, \Theta_A^{\nabla} \} = \{\Theta_A, \Theta_A \} =0$
hold.

If $(Q_{\pi} + Q_A)^2 =0$, then the sum of two structures defines a QP-structure, which provides the compatibility condition between the Poisson structure and the Lie algebroid structure.
In terms of the homological function, the conditions is equivalent to
\begin{eqnarray}
\{\Theta_{\pi}^{\nabla} + \Theta_A^{\nabla}, \Theta_{\pi}^{\nabla} + \Theta_A^{\nabla} \} =0.
\label{homological}
\end{eqnarray}
The condition $\{\Theta^{\nabla}_{\pi}, \Theta^{\nabla}_{\pi} \} =0$ holds
if we impose $R=0$ in addition to the requirement that
$\pi$ is the Poisson bivector field.

Finally we solve the condition $\{\Theta_{\pi}^{\nabla}, \Theta_A^{\nabla} \}=0$.
The concrete computation shows that this it equivalent to 
the following geometric conditions:
\begin{proposition}\label{twoP}
$\{\Theta_{\pi}^{\nabla}, \Theta_A^{\nabla} \} =0$ 
if and only if
\begin{eqnarray}
{}^A \nabla \pi &=& 0,
\label{condition01}
\\
\baS &=& 0.
\label{condition02}
\end{eqnarray}
\end{proposition}
The first condition \eqref{condition01} is same as the condition (P1) for a Hamiltonian Lie algebroid over a Poisson manifold.
The second condition \eqref{condition02} for 
the basic curvature 
plays a crucial role in maintaining consistency.
Proposition \ref{twoP} recovers Theorem 4.3 in the paper \cite{Blohmann:2023}.

Having defined the target space QP-manifold, we apply the AKSZ construction.
For general theories on the AKSZ formalism,
one may refer to \cite{Cattaneo:2001ys, Roytenberg:2006qz, Ikeda:2012pv}.
Under the condition in Proposition \ref{twoP}, 
we can construct the AKSZ sigma model for the
gauged Poisson sigma model.
In the AKSZ construction, the QP-structure on $T^*[1]A^*$ 
are mapped to the QP-structure on 
$\Map(T[1]\Sigma, T^*[1]A^*)$ via the \textit{transgression map} 
$\brT := \int_{T[1]\Sigma} \ev^*
$, which is defined as the composition of the integration over $T[1]\Sigma$
and the pullback of the evaluation map $\ev$.
For local coordinates $(x^i, z_i, w^a, y_a)$ on $T^*[1]A^*$, 
the corresponding local coordinates on $\Map(T[1]\Sigma, T^*[1]A^*)$
are obtained via the transgression map $\brT$.
These are called \textit{superfields}.
Superfields are denote $(\bX^i, \bz_i, \bw^a, \by_a)$, where 
$\bX^i: T[1]\Sigma \rightarrow M$, 
$\bz_i \in \Gamma(T[1]\Sigma, \bX^* T^*[1]M)$,
$\by_a \in \Gamma(T[1]\Sigma, \bX^* A^*)$, and 
$\bw^a \in \Gamma(T[1]\Sigma, \bX^* A[1])$.
Local coordinates $(\sigma^{\mu}, \theta^{\mu})$
on $T[1]\Sigma$ of degree $(0,1)$ are introduced.
The graded symplectic form on the mapping space 
$\Map(T[1]\Sigma, T^*[1]A^*)$ is obtained from Eq.~\eqref{gsymp} via the 
transgression map:
\begin{eqnarray}
\ggomega = \brT \gomega &=& \int_{T[1]\Sigma} \rd^2 \sigma \rd^2 \theta
\left( \delta \bX^i \wedge \delta \bz_i + \delta \by_a \wedge \delta \bw^a
\right),
\label{ggsymp}
\end{eqnarray}
where $\delta$ represents the differential on the
mapping space $\Map(T[1]\Sigma, T^*[1]A^*)$.
%
%
The AKSZ homological function $\gS$ consists of two terms, 
$\gS = \gS{}_0 + \gS{}_1$.
The first term $\gS{}_0$ arises from the transgression of 
the Liouville $1$-form 
$\vartheta$ satisfying $\gomega = - \delta \vartheta$:
\begin{eqnarray}
\vartheta &=& z_i \rd x^i + y_a \wedge \rd w^a
= z_i^{\nabla} \rd x^i + y_a \wedge \qd w^a,
\end{eqnarray}
where 
$\qd w^a = \rd w^a - \omega^a_{bi} w^b \rd x^i$.
The second term $\gS{}_1$ is given by the transgression of 
the homological function $\Theta_{\pi}^{\nabla} + \Theta_A^{\nabla}$, 
leading to the BV action functional:
\begin{eqnarray}
\gS &=& \gS{}_0 + \gS{}_1
\nonumber \\ &=& 
\int_{T[1]\Sigma} \rd^2 \sigma \rd^2 \theta
\left(\bracket{\bz^{\nabla}}{\sd \bX} 
+ \inner{\by}{\sdd \bw}
\right.
\nonumber \\ &&
\left.
+ (\pi \circ \bX)(\bz^{\nabla}, \bz^{\nabla})
- \bracket{\rho(\bw)}{\bz^{\nabla}} 
- \inner{T(\bw, \bw)}{\by} \right)
\nonumber \\
&=& \int_{T[1]\Sigma} \rd^2 \sigma \rd^2 \theta
\left(\bz_i^{\nabla} \sd \bX^i + \by_a \sdd \bw^a 
\right.
\nonumber \\ &&
\left.
+ \frac{1}{2} \pi^{ij}(\bX) \bz^{\nabla}_i \bz^{\nabla}_j
- \rho^i_a(\bX) \bz^{\nabla}_i \bw^a - \frac{1}{2} T_{ab}^c(\bX) \bw^a \bw^b \by_c \right).
\label{BVaction}
\end{eqnarray}
Here
$\sdd \bW^a = \bbd \bW^a - \omega^a_{bi}(\bX) \bW^b \bbd \bX^i$ and 
$\rd^2 \sigma \rd^2 \theta$ denotes the Berezin measure 
on the supermanifold $T[1]\Sigma$.
The homological condition $\sbv{\gS}{\gS} =0$ is satisfied
since $\sbv{\gS_0}{\gS_0}$ is trivially zero
and $\sbv{\gS_1}{\gS_1} =0$ is obtained from Eq.~\eqref{homological}
under the condition in Proposition \ref{twoP} and $R=0$.
The bracket of $\gS_0$ and $\gS_1$ is the integration of a total derivative,
and zero since $\partial \Sigma = \emptyset$:
\begin{eqnarray}
&& \sbv{\gS_0}{\gS_1} 
\nonumber \\ &&
= \left\{ \int_{T[1]\Sigma} \rd^2 \sigma \rd^2 \theta
\left(\bracket{\bz^{\nabla}}{\sd \bX} 
+ \inner{\by}{\sdd \bw}, 
(\pi \circ \bX)(\bz^{\nabla}, \bz^{\nabla})
- \bracket{\rho(\bw)}{\bz^{\nabla}} 
- \inner{T(\bw, \bw)}{\by} \right) \right\}
\nonumber \\
&&= \int_{T[1]\Sigma} \rd^2 \sigma \rd^2 \theta \
\bbd \left(
(\pi \circ \bX)(\bz^{\nabla}, \bz^{\nabla})
- \bracket{\rho(\bw)}{\bz^{\nabla}} 
- \inner{T(\bw, \bw)}{\by} \right) =0.
\end{eqnarray}
Even if the Lie algebroid $A$ is the action Lie algebroid,
this AKSZ action functional \eqref{BVaction} is different from 
the one in \cite{Zucchini:2008cg}.

The key result is summarized as follows:
\begin{theorem}\label{BVcondition}
$\sbv{\gS}{\gS} = 0$
if and only if
\begin{eqnarray}
{}^A \nabla \pi = 0, \quad \baS = 0, \qquad R=0.
\end{eqnarray}
\end{theorem}
First two conditions correspond to Proposition \ref{twoP}.

From the general theory of the BV-BFV-AKSZ formalism, 
the homological condition is same 
as the classical master equation $\sbv{\gS}{\gS} = 0$ in the BV formalism.
Thus, the classical action functional \eqref{GPSM} is gauge invariant.
In the Hamiltonian formalism, 
this means that constraints are first-class.

\subsection{GPSM with boundary}\label{sec:GPSMboundary}
In this section, we generalize the GPSM on a manifold $\Sigma$ with boundary, 
$\partial \Sigma \neq \emptyset$.
Boundary terms can be added in the action functional.
In the PSM, consistent boundary conditions are labeled 
by coisotropic submanifolds of $M$ \cite{Cattaneo:2003dp}.
Here, we assume the simplest case, 
where the coisotropic submanifold is
determined by $z_i =0$ for $i= 1, \ldots, \mathrm{dim}(M)$
Under this assumption, in our GPSM, boundary conditions must 
be consistent with the Lie algebroid structure.
A boundary condition is fixed by introducing a \textit{boundary term} in 
the action functional. 
Since the boundary is one dimension,
the boundary term is an integration of a $1$-form on $\Sigma$.
The general form is as follows:
\begin{eqnarray}
S_{\partial \Sigma} = \int_{\partial \Sigma} \bracket{X^* \mu}{W}
= \int_{\partial \Sigma} \mu_a(X) W^a,
\end{eqnarray}
where $\mu = \mu_a(x) e^a \in \Gamma(A^*)$ is a section of $A^*$.
The total action functional becomes
\begin{eqnarray}
S_b^{\nabla} &=& S^{\nabla} + S_{\partial \Sigma}
\nonumber \\
&=& \int_{\Sigma}
\left(Z^{\nabla}_i \wedge \rd X^i 
+ \frac{1}{2} \pi^{ij}(X) Z^{\nabla}_i \wedge Z^{\nabla}_j
+ Y_a \wedge DW^a - \rho^i_a(X) Z^{\nabla}_i \wedge W^a 
\right.
\nonumber \\ && 
\left.
- \frac{1}{2} T_{ab}^c(X) W^a \wedge W^b Y_c
\right) 
+ \int_{\partial \Sigma} \mu_a(X) W^a,
\label{actionb}
\end{eqnarray}
We require a consistency as a physical theory:
the total action functional is gauge invariant, 
i.e., $\delta S_b^{\nabla} =0$.
It means that the boundary condition to be consistent
with Lie algebroid structures.
Imposing this condition, we obtain conditions for the section $\mu(x)$.
Note that the gauge transformation of a superfield $\Phi$ is given by 
the equation,
\begin{eqnarray}
\delta \Phi  &=& \sbv{\gS}{\Phi}.
\label{gaugetrans}
\end{eqnarray}
Expansion of degree one coordinate $\theta^{\mu}$
on $T[1]\Sigma$ is
\begin{eqnarray}
\Phi(\sigma, \theta)  &=& \Phi^{(0)}(\sigma) + \theta^{\mu} \Phi^{(1)}_{\mu}(\sigma) 
+ \frac{1}{2} \theta^{\mu} \theta^{\nu} \Phi^{(2)}_{\mu\nu}(\sigma).
\end{eqnarray}
The explicit expansions are
\begin{eqnarray}
\bX^i(\sigma, \theta) &=& X^i(\sigma) + \theta^{\mu} Z^{+ i}_{\mu}(\sigma) 
+ \frac{1}{2} \theta^{\mu} \theta^{\nu} t^{+ i}_{\mu\nu}(\sigma),
\\
\bz_i(\sigma, \theta)  &=& - t_i(\sigma) + \theta^{\mu} Z_{\mu i}(\sigma) 
+ \frac{1}{2} \theta^{\mu} \theta^{\nu} X^{+}_{\mu\nu i}(\sigma),
\\
\bw_a(\sigma, \theta) &=& -c_a(\sigma) + \theta^{\mu} W_{\mu a}(\sigma) 
+ \frac{1}{2} \theta^{\mu} \theta^{\nu} Y^{+}_{\mu\nu a}(\sigma),
\\
\by^a(\sigma, \theta) &=& Y^a(\sigma) 
+ \theta^{\mu} W^{+ a}_{\mu}(\sigma) 
+ \frac{1}{2} \theta^{\mu} \theta^{\nu} c^{+ a}_{\mu\nu}(\sigma).
\end{eqnarray}
Using these expansions,
the gauge transformation of the action functional $S_b^{\nabla}$ becomes 
the following boundary integration,
\begin{eqnarray}
\delta S_b^{\nabla} &=& 
\int_{\partial \Sigma}
\left[(t_i - \nabla_i \mu_a c^a) \rd X^i - (\rho^i_a - \pi^{ij} \nabla_j \mu_a) t_i W^a
+ (\rho^i_a \partial_i \mu_b - C_{ab}^c \mu_c) c^a W^b
\right],
\label{boundarygauge}
\end{eqnarray}
where we have assumed the conditions in Theorem \ref{BVcondition}, i.e., 
${}^A \nabla \pi = 0, \baS = 0$ and $R=0$
which ensure the gauge invariance of the bulk action functional
except for the boundary term.
$S_b^{\nabla}$ is gauge invariant if we impost the following conditions:
\footnote{We can choose other boundary conditions to ensure gauge invariance of $S_b$. A problem of ambiguities remains future analysis.}.
\begin{eqnarray}
&& X^i = \mbox{constant}
\label{bou01}
\\
&& \rho^i_a - \pi^{ij} \nabla_j \mu_a =0,
\label{bou02}
\\
&& \rho^i_a \partial_i \mu_b - C_{ab}^c \mu_c=0.
\label{bou04}
\end{eqnarray}
The first condition Eq.~\eqref{bou01} restricts $X$ to a constant map 
on the boundary.
The second condition Eq.~\eqref{bou02} is precisely Eq.~\eqref{PMS2},
which define a momentum section over a Poisson manifold.
The third condition Eq.~\eqref{bou04} is equivalent to Eq.~\eqref{PMS3}, 
which ensures the bracket compatibility condition.
Thus, we obtain conditions (D2) and (D3) in the definition of 
a Hamilton Lie algebroid as the boundary condition of the 
GPSM \eqref{actionb}. Our result is summarized as follows:
\begin{proposition}
Assume that the condition in Theorem \ref{BVcondition} holds, this means
${}^A \nabla \pi = 0, \baS = 0$ and $R=0$.
The sigma model given by the action functional $S_b^{\nabla}$
in Eq.~\eqref{actionb} is gauge invariant 
if the target space $A$ has a Hamiltonian Lie algebroid structure 
over a Poisson manifold. 
The section $\mu$ in the boundary term is a momentum section.
\end{proposition}

\subsection{Gauged Dirac sigma models}\label{sec:GDSM}
In this section, we consider a physical theory which has a Hamiltonian Lie algebroid structure over a Dirac structure.

Let $\Xi$ be a three dimensional manifold 
with two dimensional boundary $\Sigma = \partial \Xi$ and 
let $L$ be a Dirac structure in $TM \oplus T^*M$.
The Dirac sigma model (DSM) \cite{Kotov:2004wz} is a sigma model
from $\Xi$ to a Dirac structure $L$.
Concrete construction is as follows.

Let $X:\Xi \rightarrow M$ be a map from $\Xi$ to a smooth manifold $M$.
Consider a $1$-form $V + Z$ on $\Sigma$ that takes value in the pullback 
of the subbundle $L \subset TM \oplus T^*M$,
where $V$ takes a value in $X^* TM$
and $Z$ takes a value in $X^* T^*M$.
Moreover, we assume $M$ is equipped with a metric $g$.
The action functional of the Dirac sigma model is given by
\footnote{In Section \ref{sec:GDSM}, $\alpha$ denotes a coupling constant,
is not a differential form.}
\begin{align}
S_{D} &= \int_{\Sigma} 
\left(\frac{\alpha}{2} g(\rdv X, \rdv X)
+ \bracket{Z}{\rd X} - \frac{1}{2} \bracket{Z}{V} \right)
+ \int_{\Xi} X^* H
\nonumber \\
&= \int_{\Sigma}
\left(\frac{\alpha}{2} g_{ij} \rdv X^i \wedge * \rdv X^j 
+ Z_i \wedge \rd X^i - \frac{1}{2} Z_i \wedge V^i
\right)
+ \int_{\Xi} \frac{1}{3!} H_{ijk}(X) \rd X^i \wedge \rd X^j \wedge \rd X^k,
\label{DSM}
\end{align}
where $\alpha$ is a constant, $\rdv X = \rd X - V$,
and $H$ is a closed $3$-form associated with the Dorfman bracket 
on $TM \oplus T^*M$.
The action $S_{D}$ is consistent if and only if $L$ is a Dirac structure 
of $TM \oplus T^*M$.
The DSM generalizes the PSM since,
when $H=\alpha= 0$ and $V = (\pi \circ X)^{\sharp} (Z)$, 
the action functional \eqref{DSM} reduces to that of the PSM
\cite{Kotov:2004wz}.
Here, $Z+(\pi \circ X)^{\sharp} (Z)$ is the pullback of the Dirac structure
$L_{\pi}$.

Next, we consider 'gauging' of the DSM with respect to a Lie algebroid 
$A$ over $M$, analogous to the GPSM in Section \ref{sec:GPSM}.
This is constructed by adding $S_L$ in Eq.~\eqref{GPSM} to $S_{D}$,
similar to the GPSM.
The action functional of the \textit{gauged Dirac sigma model} (GDSM) is
\begin{eqnarray}
S &=& S_D + S_L
\nonumber \\
&=& \int_{\Sigma}
\left(\frac{\alpha}{2} g_{ij} \rdv X^i \wedge * \rdv X^j 
+ Z_i \wedge \rd X^i - \frac{1}{2} Z_i \wedge V^i
\right)
+ \int_{\Xi} \frac{1}{3!} H_{ijk}(X) \rd X^i \wedge \rd X^j \wedge \rd X^k
\nonumber \\ &&
+ \int_{\Sigma} \left(
Y_a \wedge \rd W^a - \rho^i_a(X) Z_i \wedge W^a 
+ \frac{1}{2} C_{ab}^c(X) W^a \wedge W^b Y_c
\right),
\label{GDSM}
\end{eqnarray}
where the second $\Sigma$ integration term is gauging terms. 
We now rewrite Eq.~\eqref{GDSM} to a manifestly target-space covariant form:
\begin{eqnarray}
S^{\nabla} 
&=& \int_{\Sigma}
\left(\frac{\alpha}{2} g_{ij} \rdvn X^i \wedge * \rdvn X^j 
+ Z^{\nabla}_i \wedge \rd X^i - \frac{1}{2} Z^{\nabla}_i \wedge V^{\nabla i}
\right)
+ \int_{\Xi} \frac{1}{3!} H_{ijk}(X) \rd X^i \wedge \rd X^j \wedge \rd X^k
\nonumber \\ 
&& + \int_{\Sigma} \left(
Y_a \wedge \rdd W^a - \rho^i_a(X) Z^{\nabla}_i \wedge W^a 
- \frac{1}{2} T_{ab}^c(X) W^a \wedge W^b Y_c
\right),
\label{CGDSM}
\end{eqnarray}
where $\rdvn X^i = \rd X^i - V^i - \rho^i_a W^a$.

If $V = (\pi^{\sharp} \circ X)(Z)$ and $H=\alpha=0$, the action functional 
\eqref{CGDSM} reduces to the gauged Poisson sigma model \eqref{CGPSM}.
Conditions $R=0$ and ${}^AS=0$ are also required for gauge invariance 
of the action functional \eqref{CGDSM} in the GDSM.

We can adopt the same covariantized gauge transformations 
as the GPSM for the Lie algebroid action on the
fields $X^i, W^a, Y_a, Z_i$.
To verify the action of the Hamiltonian Lie algebroid $A$ on the GDSM,
we analyze the Lie algebroid action on the action functional 
$S^{\nabla}$ in \eqref{CGDSM}.
The Lie algebroid action on each field is given by $c^a$ terms 
in Eqs.~\eqref{gauge21}--\eqref{gauge24}. 
Meanwhile, the $t_i$ terms in Eqs.~\eqref{gauge21}--\eqref{gauge24} induces 
infinitesimal automorphisms on the Dirac structure, which are independent of
the Lie algebroid action.
The Lie algebroid action on $X^i, W^a, Y_a, Z_i$ is given by
the same equations as Eqs.~\eqref{gauge21}--\eqref{gauge24},
\begin{eqnarray}
\delta^{\nabla}_{c} X^i &=& \rho^i_a(X) c^a,
\label{gauge31}
\\
\delta^{\nabla}_{c} W^a &=& \rd c^a + C_{bc}^a(X) W^b c^c 
+ \omega^a_{bi} c^b \rdvn X^i
\label{gauge32}
\\
\delta^{\nabla}_{c} Y_a &=& 
- C_{ab}^c(X) Y_c c^b + \omega^c_{bi} \rho^i_a Y_c c^b
\label{gauge33}
\\
\delta^{\nabla}_{c} Z_i^{\nabla} &=& 
-\nabla_i \rho^j_a(X) Z_j^{\nabla} c^a + S_{iab}^c(X) W^a Y_c c^b.
\label{gauge34}
\end{eqnarray}
$\int_{\Xi} X^* H$ is invariant since
$\delta_c \int_{\Xi} X^* H = \int_{\Xi} X^* \iota_{\rho} \rd H =0$,
from $\rd H=0$.

The remaining terms must be gauge invariant, i.e.,
\begin{eqnarray}
&& \delta_c \int_{\Sigma}
\left(\frac{\alpha}{2} g_{ij} \rdvn X^i \wedge * \rdvn X^j
+ Z^{\nabla}_i \wedge \rd X^i - \frac{1}{2} Z^{\nabla}_i \wedge V^{\nabla i}
+ Y_a \wedge \rdd W^a - \rho^i_a(X) Z^{\nabla}_i \wedge W^a 
\right.
\nonumber \\ &&
\left.
- \frac{1}{2} T_{ab}^c(X) W^a \wedge W^b Y_c
\right) =0.
\label{remaingauge}
\end{eqnarray}
Under \eqref{gauge31}--\eqref{gauge34} and conditions $R={}^AS=0$,
Eq.~\eqref{remaingauge} holds if and only if the gauge transformation 
of $V^{\nabla}$ satisfies
\begin{eqnarray}
\delta^{\nabla}_{c} V^{\nabla i} &=& 
(\nabla_j \rho^i_a) V^{\nabla j} c^a.
\label{gauge35}
\end{eqnarray}
Additionally, the condition ${}^A \nabla g=0$ must be satisfied.
Note that, from Eq.~\eqref{gauge35},
$D_V^{\nabla} X^i$ transforms covariantly:
\begin{eqnarray}
\delta (D_V^{\nabla} X^i) &=& 
(\nabla_j \rho^i_a) (D_V^{\nabla} X^j) c^a.
\end{eqnarray}
Thus, $\delta \left(\frac{1}{2} g_{ij} \rdvn X^i \wedge * \rdvn X^j\right) =0$,
which ensures the gauge invariance of the remaining terms.

We now prove the following proposition.
\begin{lemma}\label{GDcondition}
Assume that $R = 0, \baS = 0$. Then,
Equation \eqref{gauge35} holds if and only if
condition (D1) is satisfied.
\end{lemma}

\begin{proof}
If $Z+V$ belongs to the pullback of an element in a Dirac structure $L$, 
it can be parametrized as
$Z_i = (g_{ij} (\mathrm{id} + \calO)^j_k) \mathfrak{a}^k$,
$V^i = (\mathrm{id} - \calO)^i_j \mathfrak{a}^j$,
where $\mathfrak{a} = \mathfrak{a}^i \partial_i \in \Omega^1(\Sigma, X^*TM)$
is a $1$-form on $\Sigma$ taking values in $X^*TM$, and
$\calO \in \Gamma(O(TM))$ is an orthogonal matrix with respect to 
the metric $g$ \cite{Kotov:2004wz}.
Thus, we obtain
\begin{eqnarray}
V^i &=& (\mathrm{id} - \calO)^i_j ((\mathrm{id} + \calO)^{-1})^j_k g^{kl} Z_l
= U^{ij}(x) Z_j,
\label{gauge41}
\end{eqnarray}
where 
\begin{eqnarray}
U= \frac{\mathrm{id} - \calO}{g(\mathrm{id} + \calO)}.
\label{defU}
\end{eqnarray}
Since $U$ is a local function of $x$, we obtain
\begin{eqnarray}
\delta V^{\nabla i} &=& \delta (U^{ij}(x) Z^{\nabla}_j) 
= \rho^k_a \partial_k U^{ij} Z^{\nabla}_j c^a 
+ U^{ij} \nabla_j \rho^k_a Z^{\nabla}_k c^a.
\label{gauge51}
\end{eqnarray}
Eq.~\eqref{gauge51} is equal to Eq.~\eqref{gauge35}
if and only if ${}^A \nabla U =0$,
which is equivalent to the condition (D1):
${}^A \nabla(Z+V)(e_1, e_2)$ is the pullback of $L$ 
for every $e_1, e_2 \in \Gamma(A)$.
\hfill\qed
\end{proof}
From Lemma \ref{GDcondition}, we derive the following theorem.
\begin{theorem}\label{GDcondition2}
If $R = 0, \baS = 0$, ${}^A g = 0$ and the condition (D1) holds,
then the action functional of the gauged DSM \eqref{GDSM} is gauge invariant.
\end{theorem}

\subsection{GDSM with boundary}\label{sec:GDSMboundary}
Similar to Section \ref{sec:GPSMboundary}, 
we consider the worldsheet $\Sigma$ with boundary, $\partial \Sigma \neq \emptyset$.
We introduce $\mu \in \Gamma(A^*)$ by adding a boundary term in the GDSM
and discuss relations with conditions (D2) and (D3).

For the case $H$ is not exact, we need to consider one dimensional boundaries of boundaries of the three dimensional manifold $\Xi$.
This requires working with manifolds with corners or the D-brane 
formulation of sigma models \cite{Alekseev-Schomerus}.
In this section, to avoid such technical issues,
we assume the $3$-form $H$ is globally exact, i.e.,  $H=\rd B$ for
some $2$-form $B$, for simplicity.
However, generalizing the theory to arbitrary $H$ is not difficult.
Our goal is to demonstrate that at least one physical example exists 
with a Hamiltonian Lie algebroid structure over a Dirac structure.

We add the following same boundary term as the GPSM:
\begin{eqnarray}
S_{\partial \Sigma} = \int_{\partial \Sigma} \bracket{X^* \mu}{W}
= \int_{\partial \Sigma} \mu_a(X) W^a.
\end{eqnarray}
Then, the total action functional is given by
\begin{eqnarray}
S^{\nabla}_b
&=& \int_{\Sigma}
\left(\frac{\alpha}{2} g_{ij} \rdvn X^i \wedge * \rdvn X^j 
+ Z^{\nabla}_i \wedge \rd X^i - \frac{1}{2} Z^{\nabla}_i \wedge V^{\nabla i}
+ \frac{1}{2} B_{ij}(X) \rd X^i \wedge \rd X^j
\right.
\nonumber \\ && 
\left.
+ Y_a \wedge \rdd W^a - \rho^i_a(X) Z^{\nabla}_i \wedge W^a 
- \frac{1}{2} T_{ab}^c(X) W^a \wedge W^b Y_c
\right)
+ \int_{\partial \Sigma} \mu_a(X) W^a,
\label{bGDSM}
\end{eqnarray} 
We now determine the conditions required for the total action 
to be gauge invariant, i.e., to satisfy $\delta S_b^{\nabla} =0$.

The gauge transformation of \eqref{bGDSM} is concretely computed 
using $U$ defined in Eq.~\eqref{defU} and 
gauge transformations \eqref{gauge31}--\eqref{gauge35}:
\begin{eqnarray}
\delta S_b^{\nabla} &=& 
\int_{\partial \Sigma}
\left((t_i - \nabla_i \mu_a c^a) \rd X^i - (\rho^i_a - U^{ij} \nabla_j \mu_a) t_i W^a
+ (\rho^i_a \partial_i \mu_b - C_{ab}^c \mu_c) c^a W^b
\right),
\label{boundarygauge2}
\end{eqnarray}
The action $S_b^{\nabla}$ is gauge invariant, $\delta S_b^{\nabla} =0$
if we impose the following boundary conditions:
\begin{eqnarray}
&& X^i = \mbox{constant}
\label{bou11}
\\
&& \rho^i_a - U^{ij} \nabla_j \mu_a =0,
\label{bou12}
\\
&& \rho^i_a \partial_i \mu_b - C_{ab}^c \mu_c=0.
\label{bou14}
\end{eqnarray}
Eq.~\eqref{bou11} restricts $X$ to a constant map on the boundary.
Substituting Eq.~\eqref{defU} into Eq.~\eqref{bou12},
there exists an element $\mathfrak{b} \in \Gamma(TM \otimes A^*)$ such that
\begin{eqnarray}
\nabla_i \mu_a &=& (g_{ij} (\mathrm{id} + \calO)^j_k) \mathfrak{b}^k_a,
\\
\rho^i_a &=& (\mathrm{id} - \calO)^i_j \mathfrak{b}^j_a.
\end{eqnarray}
This implies that $(\rho+\nabla \mu)(e)$ is an element of a Dirac structure, 
which is equivalent to condition (D2) in the definition of a momentum section 
over the Dirac structure.
Furthremore, Eq.~\eqref{bou14} is equivalent to Eq.~\eqref{D03} 
under condition (D2), which is identical to Eq.~\eqref{bou12}.
Thus, we obtain conditions (D2) and (D3) in the definition of 
a Hamilton Lie algebroid as the boundary condition of the GDSM \eqref{bGDSM}.
\begin{proposition}
Consider a sigma model given by the action functional \eqref{bGDSM}, 
i.e., the GDSM with boundary. Then, 
under the assumption in Theorem \ref{GDcondition2}, namely
$R = 0, \baS = 0$, ${}^A g = 0$ and (D1),
the action functional is gauge invariant 
if the boundary has a Hamiltonian Lie algebroid structure 
over a Dirac structure. 
The section $\mu$ in the boundary term is a momentum section.
\end{proposition}
If the Dirac structure is $L= L_{\pi}$, i.e.,
it is induced from a Poisson structure,
we set $V = (\pi^{\sharp} \circ X)(Z)$ in the 
GDSM with boundary \eqref{bGDSM}.
Then, the action functional then reduces to that of the 
GPSM with boundary \eqref{actionb}.

\section{Conclusion and discussion}
We have introduced the concept of a Hamiltonian Lie algebroid and a momentum section over a Dirac structure.
This framework encompasses Hamiltonian Lie algebroids over pre-symplectic and Poisson manifolds as special cases.
After deriving key formulas, we have provided several examples.
In the final section, we constructed a sigma model incorporating the structure of a Hamiltonian Lie algebroid over both a Poisson manifold and a Dirac structure.
The existence of physical models underscores the usefulness of our definition.

For further applications, it is crucial to identify more examples of Hamiltonian Lie algebroids.
For instance, an alternative gauging of the Poisson sigma model was proposed in \cite{Zucchini:2008cg} and \cite{Bonechi:2012kh},
and generalizing the concept of a momentum map to a momentum section presents an interesting direction for research.

A Lie group-valued momentum map is described by a Dirac structure known as the Cartan-Dirac structure \cite{AMM1998, CBWZ2004}.
Formulating a Lie group-valued momentum map within the framework of Hamiltonian Lie algebroids remains an important open problem.

The relationship between momentum maps and Dirac structures has been analyzed in
\cite{BursztynCavalcantiGualtieri, BursztynCrainic1, BursztynCrainic2, BursztynIglesiasPonteSevera, Balibanu-Mayrand}.
Investigating the connections between their work and our approach is an essential next step.

Another important direction is the study of the reduction of Dirac structures in analogy with symplectic reduction over a symplectic manifold \cite{Marsden-Weinstein}.

\subsection*{Acknowledgments}
\noindent
The author would like to thank to the anonymous referees for relevant contribution to improve the paper.

The author is grateful to the Institut Mittag-Leffler for support within the program ``Cohomological Aspects of Quantum Field Theory'' in 2025, where part of this work was carried, for their hospitality.
This work was supported by JSPS Grants-in-Aid for Scientific Research Number 22K03323.

\appendix
\section{Gauge transformations of the GPSM}
Gauge transformations of fields in the gauged Poisson sigma model are
\begin{eqnarray}
\delta X^i &=& - \pi^{ij}(X) t_j + \rho^i_a(X) c^a,
\label{gauge11}
\\
\delta W^a &=& \rd c^a + C_{bc}^a(X) W^b c^c,
\\
\label{gauge12}
\delta Y_a &=& - \rho^i_a(X) t_i - C_{ab}^c(X) Y_c c^b,
\label{gauge13}
\\
\delta Z_i &=& \rd t_i + \partial_i \pi^{jk}(X) Z_j t_k 
- \partial_i \rho^j_a(X) (-t_j W^a + Z_i c^a) 
+ \partial_i C_{ab}^c(X) W^a Y_c c^b,
\label{gauge14}
\end{eqnarray}
where $t_i$ is a function on $\Sigma$ taking values on $X^*T^*M$ and
$c^a$ is a function on $\Sigma$ taking values on $X^* A$.

The covariantized gauge transformations are
\begin{eqnarray}
\delta^{\nabla} X^i &=& - \pi^{ij}(X) t_j^{\nabla} + \rho^i_a(X) c^a,
\label{gauge21}
\\
\delta^{\nabla} W^a &=& \rd c^a + C_{bc}^a(X) W^b c^c 
+ \omega^a_{bi} c^b \rdd X^i
- \pi^{ij} \omega_{bi}^a (-W^b t_j^{\nabla}  + c^b Z_j^{\nabla})
\\
\label{gauge22}
\delta^{\nabla} Y_a &=& - \rho^j_a(X) t_j^{\nabla}
- C_{ab}^c(X) Y_c c^b + \omega^c_{bi} \rho^i_a Y_c c^b
- \pi^{ij} \omega_{ai}^b Y_b t_j^{\nabla},
\label{gauge23}
\\
\delta^{\nabla} Z_i^{\nabla} &=& \rd t_i^{\nabla} 
+ \partial_i \pi^{jk}(X) Z_j^{\nabla} t_k^{\nabla} 
+ \nabla_i \rho^j_a(X) (t_j^{\nabla} W^a - Z_j^{\nabla} c^a) 
\nonumber \\ && 
- \pi^{jk} R_{ija}^b W^a Y_b t_k^{\nabla}
+ \baS_{iab}^c(X) W^a Y_c c^b,
\label{gauge24}
\end{eqnarray}
where $t^{\nabla}_i := t_i - \omega_{ai}^b c^a Y_b$.

We can directly calculate the gauge transformation of the action functional $S^{\nabla}$  
under gauge transformations \eqref{gauge21}--\eqref{gauge23}.

\newcommand{\bibit}{\sl}



\end{document}